\title{\vspace{-0.7cm}Random-Player Maker-Breaker games}
\author{ Michael Krivelevich \thanks{School of Mathematical Sciences,
Raymond and Beverly Sackler Faculty of Exact Sciences, Tel Aviv
University, Tel Aviv, 6997801, Israel. Email:
krivelev@post.tau.ac.il. Research supported in part by USA-Israel
BSF Grant 2010115 and by grant 912/12 from the Israel Science
Foundation.} \and Gal Kronenberg
\thanks{School of Mathematical Sciences, Raymond and Beverly Sackler Faculty of Exact Sciences, Tel Aviv University, Tel Aviv, 6997801, Israel. Email: galkrone@mail.tau.ac.il.}}

\documentclass[12pt]{article}
\usepackage{amsmath,amssymb,latexsym,color,epsfig,enumerate,a4,dsfont,xcolor,graphicx,pict2e,tikz,caption}

\newif\ifnotesw\noteswtrue

\newtheorem{theorem}{Theorem}[section]
\newtheorem{lemma}[theorem]{Lemma}
\newtheorem{claim}[theorem]{Claim}

\newtheorem{observation}[theorem]{Observation}
\newtheorem{corollary}[theorem]{Corollary}

\newcommand{\Bin}{\ensuremath{\textrm{Bin}}}

\newenvironment{proof}{\noindent{\bf Proof\,}}{\hfill$\Box$}

\begin{document}
\maketitle

\begin{abstract}
	In a $(1:b)$ Maker-Breaker game, one of the central questions is to find the maximal value of $b$ that  allows Maker to win the game (that is, the \emph{critical bias} $b^*$). Erd\H{o}s conjectured that the critical bias for many Maker-Breaker games played on the edge set of $K_n$ is the same as if both players claim edges randomly. Indeed, in many Maker-Breaker games, ``Erd\H{o}s Paradigm" turned out to be true. Therefore, the next natural question to ask is the (typical) value of the critical bias for Maker-Breaker games where only one player claims edges randomly. A random-player Maker-Breaker
	game is a two-player game, played the same as an ordinary (biased) Maker-Breaker
	game, except that one player plays according to his best strategy and claims one element in each round, while the other plays randomly and claims $b$ (or $m$) elements. In fact, for every (ordinary) Maker-Breaker game, there are two different random-player versions; the $(1:b)$ random-Breaker game and the $(m:1)$ random-Maker game. We analyze the
	random-player version of several classical Maker-Breaker games such as the
	Hamilton cycle game, the perfect-matching game and the $k$-vertex-connectivity game (played on the edge set of $K_n$). For each of these games we find or estimate the asymptotic values of the bias (either $b$ or $m$) that allow each player to typically win the game. In fact, we provide the ``smart" player with an explicit winning strategy for the corresponding value of the bias.

\end{abstract}

\section{Introduction}
Let $X$ be a finite set and let ${\mathcal F} \subseteq 2^X$ be a
family of subsets. In the $(a:b)$ Maker-Breaker game ${\mathcal F}$,
two players, called Maker and Breaker, take turns in claiming
previously unclaimed elements of $X$. The
set $X$ is called the \emph{board} of the game and the members of
${\mathcal F}$ are referred to as the \emph{winning sets}. Maker
claims $a$ board elements per round, whereas Breaker claims $b$
elements. The parameters $a$ and $b$ are called the \emph{bias} of
Maker and of Breaker, respectively. We assume that Breaker moves
first. Maker wins the game as soon as he occupies all elements of
some winning set. If Maker does not fully occupy any winning set by
the time every board element is claimed by either of the players,
then Breaker wins the game. We say that the $(a:b)$ game ${\mathcal
	F}$ is \emph{Maker's win} if Maker has a strategy that ensures his
victory against any strategy of Breaker, otherwise the game is
\emph{Breaker's win}. The most basic case is $a=b=1$, the so-called
\emph{unbiased} game, while for all other choices of $a$ and $b$ the
game is called \emph{biased}.

It is natural to play Maker-Breaker games on the edge set of a graph
$G=(V,E)$. In this case, $X=E$ and the winning sets are all edge
sets of subgraphs of $G$ which possess some given  graph property
$\mathcal P$. In this case, we refer to this game as the $(a:b)$
$\mathcal P$-game. In the special case where $G=K_n$ we denote
$\mathcal P_n:=\mathcal P(K_n)$. In the {\em connectivity game},
Maker wins if and only if his edges contain a spanning tree of $G$.
In the \emph{perfect-matching} game the winning sets are all sets containing
$\lfloor |V(G)|/2 \rfloor$ independent edges of $G$. Note that if
$|V(G)|$ is odd, then such a matching covers all vertices of $G$ but
one. In the \emph{Hamiltonicity game} the winning sets are all edge
sets containing a Hamilton cycle of $G$. 
In the \emph{$k$-connectivity game} the winning sets are all edge sets of
$k$-vertex-connected spanning subgraphs of $G$.

Playing unbiased Maker-Breaker games on the edge set of $K_n$ is
frequently in a favor of Maker. For example, it is easy to see (and
also follows from \cite{Lehman}) that for every $n\geq 4$, Maker can
win the unbiased connectivity game in $n-1$ moves (which is clearly
also the fastest possible strategy). Other unbiased games played on
$E(K_n)$ like the perfect-matching game, the Hamiltonicity game, the
$k$-vertex-connectivity game and the $T$-game where $T$ is a given
spanning tree with bounded maximum degree, are also known to be easy wins for Maker (see e.g., \cite{CFGHL,FH,HKSS2009b}). It is thus
natural to give Breaker more power by allowing him to claim $b>1$
elements in each turn.

Given a monotone increasing graph property $\mathcal P$, it is easy
to see that the Maker-Breaker game $\mathcal P(G)$ is \emph{bias
	monotone}. That is, none of the players can be harmed by claiming
more elements. Therefore, it makes sense to study $(1:b)$ games and
the parameter $b^*$ which is the \emph{critical bias} of the game,
that is, $b^*$ is the maximal bias $b$ for which Maker wins the
corresponding $(1:b)$ game $\mathcal F$.

The most fundamental question in the field of Maker-Breaker games is finding the value of the critical bias $b^*$. Erd\H{o}s suggested the following (rather unexpected) approach which has become known as the ``probabilistic intuition" or the ``Erd\H{o}s Paradigm". Consider the $(1:b)$ Maker-Breaker game $\mathcal P_n$ where $\mathcal P$ is a monotone graph property. Then according to the intuition, the maximum value of $b$ which allows Maker to win the game playing according to his best strategy should be approximately the same as the maximum value of $b$ for which Maker is the typical winner where both players play \emph{randomly}. Observe that if indeed both players play randomly, then the resulting graph (of Maker) is distributed according to the well-known random graph model $\mathcal {G}(n,m)$ for $m\sim \frac {\binom n2}{b+1}$, and thus we can estimate the value of $b^*$, according to the Erd\H{o}s Paradigm, by having the value of the threshold function $m^*$ for the same graph property $\mathcal P$. 

In several important Maker-Breaker games Erd\H{o}s Paradigm turned out to be true.  
 For example, Chv\'atal and Erd\H{o}s \cite{CE} showed
that for every $\varepsilon >0$, playing with bias
$b=\frac{(1+\varepsilon)n}{\ln n}$, Breaker can isolate a vertex in
Maker's graph while playing on the board $E(K_n)$. It thus follows
that with this bias, Breaker wins every game for which the winning
sets consist of subgraphs of $K_n$ with positive minimum degree, and
therefore, for each such game we have that $b^*\leq
\frac{(1+o(1))n}{\ln n}$. Much later, Gebauer and Szab\'o showed in
\cite{GS} that the critical bias for the connectivity game played on
$E(K_n)$ is indeed asymptotically equal to $\frac n{\ln n}$. In a
relevant development, the first author of this paper proved in
\cite{Ham} that the critical bias for the Hamiltonicity game is
asymptotically equal to $\frac{n}{\ln n}$ as well. Indeed, the critical bias in all of the above results (and more) corresponds to the threshold function for the same properties in the random graph model (see e.g., \cite{Bol}). We refer the
reader to \cite{Beck,PG} for more background  on the Erd\H{o}s Paradigm, positional games in
general and Maker-Breaker games in particular.

This probabilistic intuition relates the fields of positional games and random graphs.
Therefore, it is natural to study Maker-Breaker games that involve randomness.
One such version of Maker-Breaker games is the random-turn Maker-Breaker
games. A \emph{$p$-random-turn Maker-Breaker game} is the same as an
ordinary Maker-Breaker game, except that instead of alternating
turns, before each turn a biased coin is being tossed and Maker
plays this turn with probability $p$ independently of all other
turns. Maker-Breaker games under this setting were 
considered in \cite{PSSW} and in \cite{RandomTurn}.

In this paper we consider a different (randomized) version of Maker-Breaker
games. Since the Erd\H{o}s Paradigm relates biased Maker-Breaker games with biased Maker-Breaker games where both players play randomly, a natural question  to ask is how the critical bias changes when \emph{only one} player plays randomly.  In the  $(m:b)$ \emph{random-player Maker-Breaker game} one of the
players plays according to his best
strategy and claims exactly one element in each round, while the other player claims in every round
$b$ elements, if he is Brekaer (or $m$ if he is Maker), uniformly among all unclaimed edges.

Clearly, this (random) version of the bias Maker-Breaker games is also bias monotone. More explicitly, if for some bias $b$, it is known that Breaker w.h.p.\ wins the $(1:b)$ random-Breaker game with respect to some monotone increasing graph property $\mathcal P$, then w.h.p.\ Breaker will also win the $(1:(b+1))$ random-Breaker game (a similar statement holds for the random-Maker game). Therefore, it makes sense to study $(1:b)$ (respectively, $(m:1)$) games and
the parameter $b^*$ (respectively, $m^*$) which is the \emph{critical bias} of the game,
that is, the maximal bias for which the ``smart" player  w.h.p.\ (see Section 1.1 for a formal definition of this notion) wins the
corresponding random-player game $\mathcal F$.

Maker-Breaker games for which one player plays randomly have already been implicitly discussed: Bednarska  and  \L uczak \cite{BL} showed that in the $(1:b)$ (ordinary) Maker-Breaker games on $E(K_n)$, where Maker's goal is to build a copy of a fixed graph $H$, the ``random strategy" is nearly optimal for Maker. In particular, they proved that if $H$ is a fixed graph with at least 3 non-isolated vertices, then the critical bias for the $H$-game is $b^*=\Theta (n^{1/m(H)})$, where $m(H)=\max _{\overset{H'\subseteq H}{v(H')\geq 3}}\{\frac {e(H')-1}{v(H')-2}\}$. In fact, they showed that if in each turn, Maker claims one element \emph{randomly} and $b\leq cn^{1/m(H)}$ for some constant $c>0$, then Maker is the typical winner of the game and therefore there exists a deterministic winning strategy for Maker for these values of $b$. 

In this paper, we study the critical bias of the random-player version for some well-known Maker-Breaker games played on the edge set of a complete graph. Furthermore, in cases where the typical winner is ``smart", we present strategies that w.h.p.\ are winning strategies for the game. We do this for both random-Maker and random-Breaker games.

 In the $(1:b)$ random-Breaker game, $(X,\mathcal F)$, there are two players, Maker and Breaker. In each round, Breaker claims $b$ elements from the board, chosen independently uniformly at random among all unclaimed elements, and then Maker claims one element from the board (according to his best strategy). Clearly, the critical bias for the random-Breaker games in bounded from below by the critical bias in the ordinary Maker-Breaker games. We also have a trivial upper bound for the value of $b^*$ which is the value of $b$ that, independent of the course of the game, does not allow Maker's  graph to achieve the desired property, due to trivial graph-theoretic reasons.  For example, in the Hamiltonicity game, Maker needs at least $n$ edges in his graph and thus $b^*\leq \frac n2$. We show that in the random-Breaker games, the upper bound is actually the correct value of $b^*$ in several well-studied games played on the edge set of $K_n$. 
 
 Let $\mathcal{H}_n$ be the game played on $E(K_n)$ where
 Maker's goal in to build a Hamilton cycle. The following theorem states that if Breaker is the random player and claims $b\leq(1-\varepsilon)\frac n2$ edges in each round, then Maker can typically win the Hamiltonicity game. Clearly this result is asymptotically tight, since for $b\geq(1+\varepsilon)\frac n2$, after claiming all the edges in the graph, Maker has less than $n$ edges.

 \begin{theorem}\label{thm:RBHam}
 	Let $\varepsilon>0$, let $n$ be an integer and let
 	$b\leq(1-\varepsilon)\frac n2$. Then Maker has a 
 	strategy which is w.h.p.\ a winning strategy for the random-Breaker $(1:b)$ $\mathcal{H}_n$ game in $(1+o(1))n$
 	rounds.
 \end{theorem}

 Under the same setting, let $\mathcal{PM}_n$ be the random-Breaker game played on $E(K_{n})$ where
 Maker's goal in to build a graph which contains a perfect matching (or \emph{a nearly} perfect matching -- if $n$ is odd). In the following theorem we prove that if Breaker plays randomly and claims $b\leq(1-\varepsilon)  n$ edges in each round, then Maker typically wins the perfect-matching game. Clearly this result is also asymptotically tight, since for $b\geq(1+\varepsilon)  n$, after claiming all  edges, Maker has less than $\frac n2$ edges.

 \begin{theorem}\label{thm:RB-PM}
 	Let $\varepsilon>0$, let $n$ be an integer and let
 	$b\leq(1-\varepsilon)  n$. Then Maker has a 
 	strategy which is w.h.p.\ a winning strategy for the $(1:b)$ random-Breaker $\mathcal{PM}_n$ game in $(1+o(1))\frac n2$
 	rounds.
 \end{theorem}

Let $k$ be an integer and let $\mathcal{C}^k_n$ be the random-Breaker game played on $E(K_n)$ where
 Maker's goal is to build a $k$-connected graph on $n$ vertices. Using Theorem~\ref{thm:RBHam} and Theorem~\ref{thm:RB-PM}, in the following theorem we prove that if Breaker plays randomly and claims $b\leq(1-\varepsilon) \frac nk$ edges in each round, then Maker typically wins the $k$-connectivity game.  This result is asymptotically tight, since for $b\geq(1+\varepsilon) \frac nk$, after claiming all  edges in the graph, the average degree in Maker's graph is smaller than $k$ and therefore the minimum degree is also smaller than $k$.

 \begin{theorem}\label{thm:RB-C^k}
 	Let $\varepsilon>0$, let $n$ be an integer and let
 	$b\leq(1-\varepsilon)\frac nk$. Then Maker has a 
 	strategy which is w.h.p.\ a winning strategy for the $(1:b)$ random-Breaker $\mathcal{C}^k_n$ game in $(1+o(1))\frac {kn}2$
 	rounds.
 \end{theorem}

The other version of random-player Maker-Breaker games is the random-Maker games. Unlike the ordinary Maker-Breaker games, for several standard games in this version it turns out to be rather difficult for Maker to win the game. Even in the unbiased version $(1:1)$, it turns out that in many games Breaker is the typical winner of the game. Therefore, it makes sense to study the $(m:1)$ random-Maker games and to look for the critical bias of Maker.
In the $(m:1)$ random-Maker games, $(X,\mathcal F)$, there are two players, Maker and Breaker. In each round, Maker claims $m$ elements from the board, chosen independently uniformly at random among all unclaimed elements, while Breaker claims one element from the board (according to his best strategy).   In this case, the critical bias of the game, $m^*$, is the maximal value of $m$ for which Breaker is the typical winner of the game.

The first and the most basic game we discuss is the game where Breaker's goal is to {\em isolate a vertex} in Maker's graph playing on $E(K_n)$.  Recall the result of
Chv\'atal and Erd\H os \cite{CE} about isolating a vertex in biased Maker-Breaker game. In the following theorem we show that playing a
random-Maker game on $E(K_n)$, Breaker has a strategy
that typically allows him to isolate a vertex in Maker's graph,
provided that $m=O(\ln \ln n)$. It thus follows that for
this range of $m$, Breaker typically wins every game whose winning
sets consist of spanning subgraphs with a positive minimum degree
(such as the Hamiltonicity game, the perfect-matching game, the $k$-connectivity game,  etc.).

\begin{theorem}\label{thm:RMIsoVertex}
	Let  $\varepsilon>0$, let $n$ be an integer  and let $m\leq (\frac 12-\varepsilon)\ln
	\ln n$. Then w.h.p.\ Breaker has a strategy to isolate a vertex in Maker's
	graph while playing the $(m:1)$-random-Maker game.
\end{theorem}

Our next theorem shows that in the random-Maker Hamiltonicity game, if $m=\Omega (\ln \ln n)$ then Maker is the typical winner of the game. Together with Theorem
\ref{thm:RMIsoVertex}, this implies that for the random-Maker Hamiltonicity game, $m^*=\Theta(\ln \ln n)$.

\begin{theorem}\label{thm:RMHam}
There exists a constant $A>0$ such that if $m\geq A\ln
	\ln n$, then w.h.p.\ Maker's graph contains a Hamilton cycle while
	playing the $(m:1)$-random-Maker game.
\end{theorem}

Finally, let $k$ be an integer and consider the $(m:1)$ \emph{random-Maker
	$k$-vertex-connectivity game} played on the edge set of $K_n$, where
Maker's goal is to build a spanning subgraph which is
$k$-vertex-connected.  In the following
theorem we show that for $m=\Omega_k (\ln \ln n)$, Maker is the typical winner of this game. Again, together with Theorem
\ref{thm:RMIsoVertex} we have that in the random-Maker $k$-connectivity game, $m^*=\Theta_k(\ln\ln n)$.

 \begin{theorem}\label{thm:RMCon}
 	For every integer $k>0$, there exists a constant  $A>0$ such that if $m\geq A\ln
 	\ln n$, then playing the $(m:1)$-random-Maker game, w.h.p.\ Maker's graph is $k$-connected.
 \end{theorem}

\subsection{Notation and terminology}

Our graph-theoretic notation is standard and follows that of \cite{West}. In particular we use the following:

For a graph $G$, let $V=V(G)$ and $E=E(G)$ denote its set of
vertices and edges, respectively. For subsets $U,W\subseteq V$ we
denote by $E_G(U,W)$  all the edges
$e\in E$ with both endpoints in $U\cup W$ for which $e\cap U\neq
\emptyset$ and $e\cap W\neq \emptyset$. 

Playing Maker-Breaker game where the board $X$ is the edge set of some graph $G$, we denote by $M$ the subgraph of $G$ consisting of Maker's edges, at any point during the game. Similarly, we denote by $B$ the graph of Breaker and $F=G\setminus(M\cup B)$ is the subgraph of all unclaimed edges, at any point during the game. We say that an edge $e\in E$ is \emph{available} if $e\in E(F)$.

We also denote by $E_M(U,W)$
(respectively, $E_B(U,W)$) all such edges claimed by Maker
(respectively, Breaker) and by $E_F(U,W)$ all such unclaimed edges. We let $e(U,W)$ denote the number of edges in $E(U,W)$ (respectively, $e_M(U,W)$, $e_B(U,W)$ and $e_F(U,W)$ are the number of edges in $E_M(U,W)$, $E_B(U,W)$ and $E_F(U,W)$). For a  subset $U\subset V$ , we write
$N_G(U) = \{v \in V \setminus U : \exists u\in U\ s.t.\ \{u,v\}\in
E(G)\}$ and $N_M(U) = \{v \in V \setminus U : \exists u\in U\ s.t.\
\{u,v\}\in E(M)\}$ (or $N_B(U)$).

For a graph $G=(V,E)$ let $\overline{G}=(\overline{V},\overline{E})$ denote the complement graph of $G$, that is, $\overline{V}=V$ and $\overline{E}=\{\{u,v\}~|~u\neq v\in V,~ \{v,u\}\notin E\}$. We also write $\Delta(G)$ for the maximum degree in $G$. For a set of vertices $U\subseteq V$, we denote by $G[U]$ the corresponding vertex-induced subgraph of $G$ and we denote by $E_G[U]$ the edges of $G[U]$.

We assume that $n$ is large enough where needed.  We say that an event holds \emph{with high probability} (w.h.p.) if its probability  tends to one as $n$ tends to infinity.  For the sake of simplicity and clarity of presentation, and in order
to shorten some of the proofs, no real effort is made to optimize
the constants appearing in our results. We also sometimes omit floor
and ceiling signs whenever these are not crucial.

\section {Tools}

\subsection{Binomial and Hypergeometric distribution bounds}
We use extensively the following standard bound on the lower and the
upper tails of the Binomial distribution due to Chernoff (see, e.g.,
\cite{AloSpe2008}, \cite{JLR}):

\begin{lemma}\label{Che}
	Let $X\sim \Bin(n,p)$ and $\mu=\mathbb{E}(X)$, then
	\begin{enumerate}
		\item $\Pr\left(X<(1-a)\mu\right)<\exp\left(-\frac{a^2\mu}{2}\right)$ for every $a>0.$
		\item $\Pr\left(X>(1+a)\mu\right)<\exp\left(-\frac{a^2\mu}{3}\right)$ for every $0 < a < 1.$
	\end{enumerate}
\end{lemma}

 Let $HG(N,K,n)$ be the Hypergeometric distribution with parameters $N$, $K$ and $n$, where $N$ is the size of the population containing exactly $K$ successes and $n$ is the number of draws. The following lemma is a Chernoff-type bound on the upper and lower tails of the Hypergeometric distribution.

\begin{lemma}\label{CheHyper}
	\label{chernoff_hypg}
	Let $N\ge 0$, and let $0\le K,n\le N$ be natural numbers. Let
	$X\sim HG(N,K,n)$, $\mu=\mathbb E[X]=nKN^{-1}$. Then, inequalities
	1 and 2 from Lemma \ref{Che} hold.
\end{lemma}

\subsection{Properties of graphs and subgraphs}

First we state a standard fact about subgraphs of large minimum degree. We use this observation in the proof of Theorem \ref{thm:RBHam}.

\begin{observation}\label{re:SubGraphMinDeg} [See, e.g., Ex. 1.3.44 in \cite{West}]
Let $r>0$, then every graph with average degree at least $2r$ contains a subgraph with minimum degree at least $r+1$.
\end{observation}

The next two  claims are used to prove Theorem \ref{thm:RB-PM}. In the claims we consider a bipartite graph $G$ satisfying some pseudo-random properties.

\begin{claim}\label{claim:TisSmall}
	Let $0<\varepsilon,\alpha<1$ be constants and let $G=(A_0\cup A_1,E)$ be a bipartite graph with parts of size $n$, satisfying the following property:
		For every $X_0\subseteq A_0$, $X_1\subseteq A_1$ such that $|X_0|= n^{\alpha}$, $|X_1|= n^{\alpha/2}$, we have $e(X_0,X_1)\geq {\varepsilon} |X_0|\cdot |X_1|$. 
	
	Then for every two subsets $U_i\subseteq A_i$, $|U_0|=|U_1|=n^{\alpha}$ the following holds ($i\in \{0,1\}$). 
	\begin{enumerate}[(a)]
		\item  The sets of vertices  $T_i=\{v\in U_i~|~e(v, U_{1-i})<\frac {\varepsilon}2 n^{\alpha}\}$ are of size less than $n^{\alpha /2}$.
		\item In every set $W_i\subseteq A_i$ of size $\frac  {\varepsilon}5n$, there is a vertex $w\in W_i$ such that $e(w,U_{1-i})\geq \frac {\varepsilon}2 n^{\alpha}$.
	\end{enumerate}
\end{claim}

 
 \begin{proof}
 	For item $(a)$,  if $|T_1|\geq n^{\alpha /2}$,  look at the subset $T'\subset T_1$, $|T'|=n^{\alpha/2}$. Then, $e(T',U_2)\geq {\varepsilon} |T'|\cdot |U_2|={\varepsilon}n^{3\alpha /2}$. But, $e(T',U_2)= \sum_{v\in T'}e(v,U_2)<|T'|\cdot \frac {\varepsilon}2 n^{\alpha}=\frac {\varepsilon}2 n^{3\alpha/2}$ -- a contradiction. Therefore, $|T_1|<n^{\alpha/2}$. The proof for $T_2$ is similar. Item $(b)$ follows immediately from $(a)$.
 \end{proof}
 
 \bigskip
 
 In the next claim we show a version of Hall's condition.
 
 \begin{claim}\label{claim:PMinR}
 	Let $\varepsilon>0$ and let $G(A_0\cup A_1,E)$ be a bipartite graph with parts of size $n$. Assume that $G$ satisfies the following properties:
 	\begin{enumerate}
 		\item for every $v\in A_i$ ($i\in \{1,2\}$), $d(v)\geq  {\varepsilon} n$, and
 		\item For every $X_0\subseteq A_0$, $X_1\subseteq A_1$ such that $|X_i|= \varepsilon n$,  we have $e(X_0,X_1)\geq 1$.
 	\end{enumerate} 
 	Then $G$ contains a perfect matching.
 \end{claim}  
 
 \begin{proof}
 	We need to show that for every $X\subseteq A_0$ we have $|N(X)|\geq |X|$ (Hall's condition). First, let $X\subseteq A_0$ be a set of size $|X|\leq  {\varepsilon} n$. Then by item 1, $|N(X)|\geq |X|$. Next, if $|X|> (1- {\varepsilon})n$ then according to item 1, $|N(X)|=n$ and we are done.
 	Finally, assume $ {\varepsilon} n<|X|\leq (1- {\varepsilon})n$ but $|N(X)|<|X|$. But then $|N(X)|\leq (1- {\varepsilon})n$ and thus for $Z=A_1\setminus N(X)$, we have $|Z|>  {\varepsilon}n$ and from item 2 we have $e(X,Z)\geq 1$. This is a contradiction.
 \end{proof}

\subsection{Expanders}

For positive constants $R$ and $c$, we say that a graph $G=(V,E)$ is
an $(R,c)-expander$ if $|N_G(U)|\geq c|U|$ holds for every
$U\subseteq V$, provided $|U|\leq R$. When $c=2$ we sometimes refer
to an $(R,2)$-expander as an $R$-expander. Given a graph $G$, a
non-edge $e = \{u,v\}$ of $G$ is called a \emph{booster} if adding $e$ to
$G$ creates a graph $G'$ which is Hamiltonian, or contains a path
longer than a maximum length path in $G$.

The following lemma states that if $G$ is a ``good enough" expander,
then it is also a $k$-vertex-connected graph.

\begin{lemma}\label{lemma:connectivity} {\bf [Lemma 5.1 from \cite{BFHK}]}
	For every positive integer $k$, if $G = (V,E)$ is an $(R,
	c)-expander$ with $c \geq k$ and $Rc \geq \frac 12(|V | + k)$, then
	$G$ is $k$-vertex-connected.
\end{lemma}

The next lemma due to P\'osa (a proof can be found for example in
\cite{Bol}), shows that every connected and non-Hamiltonian expander
has many boosters.

\begin{lemma}\label{lemma:boosters}
	Let $G = (V, E)$ be a connected and non-Hamilton $R$-expander. Then
	$G$ has at least $\frac{(R+1)^2}{2}$ boosters.
\end{lemma}

The following standard lemma shows that in expander graphs, the sizes of
connected components cannot be too small.

\begin{lemma}\label{lemma:component}
	Let $G = (V, E)$ be an $(R,c)$-expander. Then every connected
	component of $G$ has size at least $R(c+1)$.
\end{lemma}

\begin{proof}
	Assume towards a contradiction that there exists a connected
	component of size less than $R(c+1)$. Let $V_0\subset V$ be the
	vertex set of this component. Choose an arbitrary subset $U\subseteq
	V_0$ such that $|U|=min\{R,|V_0|\}$. Since $G$ is an
	$(R,c)$-expander and $|U|\leq R$, it follows that $|N_G(U)|\geq
	c|U|$. Moreover, note that $N_G(U)\subseteq V_0$ as $V_0$ is a
	connected component,  therefore
	$$|V_0|\geq |U|+|N_G(U)|\geq |U|+c|U|=(c+1)|U|,$$
	which implies $|U|\leq \frac{|V_0|}{c+1}$. On the other hand,
	since $|V_0|<R(c+1)$ and $|U|=\min\{R,|V_0|\}$, it follows that
	$|U|>\frac {|V_0|}{c+1}$, which is clearly a contradiction.
\end{proof}

 \section{Random Breaker games}
 
 In this section we consider the random-player setting where Breaker plays randomly and in every round claims $b$ elements independently at random, chosen from all available elements. Here we prove Theorems \ref{thm:RBHam}, \ref{thm:RB-PM} and \ref{thm:RB-C^k}.

 \subsection{Random Breaker Hamiltonicity game}
 
 In this section we prove Theorem~\ref{thm:RBHam}.

Since this game is bias monotone, we can assume that $b= (1-\varepsilon)\frac n2$.   First we present a strategy for Maker and then prove that during the
 game he can typically follow this strategy. For this,  recall that by $B$ (respectively $M$), we denoted Breaker's (or Maker's) graph at any point during
 the game. We say that some vertex
 $v$ is \emph{free} if for every edge $e\in E(M)$, $v \notin e$.
 
 {\bf Strategy S$_{Ham}$:} Maker's strategy is divided into two stages.
 
 {\bf Stage I:} In this stage Maker's goal in to build a path of
 length $n-n^{1/4}$. For the sake of the argument, Maker thinks of his path at this stage as being directed; the directions will be ignored at later stages. Denote by $R$ the set of vertices that are not in Maker's
 	path.
 \begin{itemize}
 	\item Step 1: After Breaker's first move, Maker chooses an available edge $\{v_0,v_1\}$ such that $e_F(\{v_1\},R)\geq n^{1/5}$. Then Maker updates $P\leftarrow \{\overset{\longrightarrow}{v_0,v_1}\}$.\\
 	
 	In the $k^{th}$ round  ($2\leq k\leq {n-n^{1/4}}$), after a (random) move by Breaker, Maker acts as follows. For every $w\in V$, let $R_{w}=\{\{w,u\}\in E\ |\ u\notin P\ \text{and} \ \{w,u\}\ \text{is available}\}$
 	
 	\item Step $k$: Let $v$ be the last vertex in the (directed) path $P$. Maker finds a vertex $u\in R_{v}$ such that $|R_u|\geq  n^{1/5}$. Then Maker claims the edge $\{v,u\}$ and updates $P\leftarrow P\cup \{\overset{\longrightarrow}{v,u}\}$. 
 \end{itemize}
 The procedure stops when Maker can no longer follow the
 strategy or after $n-n^{1/4}$ times. Following
 this strategy, at the end of this stage, $\Delta (M)\leq 2$.
 
 {\bf Stage II:} In this stage Maker increases his path vertex by
 vertex. 
 \begin{itemize}
 	\item {\bf Step 1:} Maker looks at the endpoints, $v_0,v_s$, of his
 	path. If there is an available edge in $E(v_0,R)$ or in
 	$E(v_s,R)$ then Maker claims this edge and repeats Step 1 (after Breaker's move). Otherwise,
 	all of the endpoints of available edges incident to $v_0$ and $v_s$ are in the
 	path. Let $X_0$ and $X_s$ be the sets of available edges incident to $v_0$ and
 	$v_s$, respectively. We split now these sets into 4 sets: $X_0=Y_0\cup Z_0$ and
 	$X_s=Y_s\cup Z_s$ where $||Y_0|-|Z_0||\leq 1$, $||Y_s|-|Z_s||\leq 1$
 	and $Y_0$ (respectively, $Y_s$) are the edges whose other endpoints are closer to $v_0$ (respectively, $v_s$)
 	on the path. In the next 3 turns of Maker, he closes his path to a cycle as follows (see Figure \ref{fig:rotation} for an illustration).
 	
 	{\bf Case 1:} If all  endpoints of $Y_0$ (other than $v_0$) come before all
 	endpoints of $Y_s$ (other than $v_s$) in the path, then Maker finds an edge $\{v_0,v_i\}$ in $Y_0$
 	and another edge $\{v_j,v_s\}$ in $Y_s$, such that the edge $\{v_{i-1},v_{j+1}\}$ is available  and allows him to close
 	his path to a cycle in 3 steps (the vertices of the cycle are the
 	same as the vertices in the path).
 	
 	{\bf Case 2:} If all endpoints of $Z_s$ (other than $v_s$) come before all
 	endpoints of $Z_0$ (other than $v_0$) in the path, then Maker chooses an edge  $\{v_0,v_j\}$ in $Z_0$
 	and another edge $\{v_i,v_s\}$ in $Z_s$, such that the edge $\{v_{i-1},v_{j+1}\}$ is available  and allows him to close
 	his path to a cycle in 3 steps (the vertices of the cycle are the
 	same as the vertices in the path).
 	
 	In his next 3 turns, Maker claims the three edges and closes his path to a cycle, as described above. Then, Maker
 	continues to Step 2.

\begin{figure}[!h]
	\centering
	
	\begin{tikzpicture}[xscale=1.0,yscale=1.0]
	\tikzset{vertex/.style={fill,circle,inner sep=1pt}}

	\node at (80pt,50pt) {\textbf{Case 1:}};
	
	\node [font=\scriptsize] at (5pt,20pt){$Y_0$};
	\node [font=\scriptsize] at (165pt,20pt){$Y_s$};
	
	\node at (360pt,50pt) {\textbf{Case 2:}};
	
	\node [font=\scriptsize] at (290pt,20pt){$Z_0$};
	\node [font=\scriptsize] at (450pt,20pt){$Z_s$};

	\fill[gray!50] (2,0) ellipse (0.5 and 0.15);
	
	\fill[gray!50] (4,0) ellipse (0.5 and 0.15);
	
	\begin{scope}
	\node[vertex][label={[font=\scriptsize]below:\(v_0\)}] (v0)  at (0,0) {};
	
	\node[vertex][label={[font=\scriptsize]below:\(v_i\)}] (vi) at (2,0) {};
	\node[vertex][label={[font=\scriptsize]below:\(v_{j}\)}] (vj) at (4,0) {};
	
	\node[vertex][label={[font=\scriptsize]below:\(v_s\)}] (vs) at (6,0) {};
	
	\end{scope}
	\path (v0) --  (vi);
	\draw[black] (v0) -- (vs);
	
	\draw  (0,0) .. controls (0.7,1) and (1.6,1) .. (2.5,0);
	
	\draw  (0,0) .. controls (0.5,0.6) and (1,0.6) .. (1.5,0);
	
	\draw  (3.5,0) .. controls (4.4,1) and (5.3,1) .. (6,0);
	
	\draw  (4.5,0) .. controls (5,0.6) and (5.5,0.6) .. (6,0);

	------
	\fill[gray!50] (12,0) ellipse (0.5 and 0.15);
	
	\fill[gray!50] (14,0) ellipse (0.5 and 0.15);
	
	\begin{scope}
	\node[vertex][label={[font=\scriptsize]below:\(v_0\)}] (v0) at (10,0) {};
	
	\node[vertex][label={[font=\scriptsize]below:\(v_i\)}] (vi) at (12,0) {};
	\node[vertex][label={[font=\scriptsize]below:\(v_{j}\)}] (vj) at (14,0) {};
	
	\node[vertex][label={[font=\scriptsize]below:\(v_s\)}] (vs) at (16,0) {};
	
	\end{scope}
	\path (v0) --  (vi);
	\draw[black] (v0)  -- (vs);
	
	\draw  (10,0) .. controls (10.5,1) and (13,1) .. (14.5,0);
	
	\draw  (10,0) .. controls (10.5,0.6) and (12.2,0.6) .. (13.5,0);
	
	\draw (11.5,0) .. controls (12.7,1) and (15.5,1) .. (16,0);

	\draw (12.5,0) .. controls (13.2,0.6) and (15.5,0.6) .. (16,0);
	
	\end{tikzpicture}
	
	\vspace{5mm}
	
	\begin{tikzpicture}[xscale=1.0,yscale=1.0]
	\tikzset{vertex/.style={fill,circle,inner sep=1pt}}
	
	\begin{scope}
	\node[vertex][label={[font=\scriptsize]below:\(v_0\)}] (v0) at (0,0) {};
	\node[vertex][label={[font=\scriptsize,xshift=-2mm, yshift=0mm]below:\(v_{i-1}\)}] (vi-1) at (1.3,0) {};
	\node[vertex][label={[font=\scriptsize]below:\(v_i\)}] (vi) at (2,0) {};
	\node[vertex][label={[font=\scriptsize]below:\(v_{j}\)}] (vj) at (4,0) {};
	\node[vertex][label={[font=\scriptsize,xshift=3mm, yshift=0mm]below:\(v_{j+1}\)}] (vj+1) at (4.7,0) {};
	\node[vertex][label={[font=\scriptsize]below:\(v_s\)}] (vs) at (6,0) {};
	
	\end{scope}
	\path (v0) --  (vi);
	\draw[dashed] (v0) -- (vi-1) -- (vi) -- (vs);
	\draw[  black] (v0) -- (vi-1); 
	\draw[  black] (vi) -- (vj);
	\draw[  black] (vj+1) -- (vs);
	
	\draw  (0,0) .. controls (0.5,1) and (1.2,1) .. (2,0);
	
	\draw (4,0) .. controls (4.7,1) and (5.5,1) .. (6,0);
	
	\draw (1.3,0) .. controls (2,-1.2) and (4,-1.2) .. (4.7,0);
	
	------

	\begin{scope}
	\node[vertex][label={[font=\scriptsize]below:\(v_0\)}] (v0) at (10,0) {};
	\node[vertex][label={[font=\scriptsize,xshift=-2mm, yshift=0mm]below:\(v_{i-1}\)}] (vi-1) at (11.3,0) {};
	\node[vertex][label={[font=\scriptsize]below:\(v_i\)}] (vi) at (12,0) {};
	\node[vertex][label={[font=\scriptsize]below:\(v_{j}\)}] (vj) at (14,0) {};
	\node[vertex][label={[font=\scriptsize,xshift=3mm, yshift=0mm]below:\(v_{j+1}\)}] (vj+1) at (14.7,0) {};
	\node[vertex][label={[font=\scriptsize]below:\(v_s\)}] (vs) at (16,0) {};
	
	\end{scope}
	\path (v0) --  (vi);
	\draw[dashed] (v0) -- (vi-1) -- (vi) -- (vs);
	\draw[ black] (v0) -- (vi-1); 
	\draw[  black] (vi) -- (vj);
	\draw[  black] (vj+1) -- (vs);
	
	\draw  (10,0) .. controls (10.5,1) and (13.2,1) .. (14,0);
	
	\draw (12,0) .. controls (12.7,1) and (15.5,1) .. (16,0);
	
	\draw (11.3,0) .. controls (12,-1.2) and (14,-1.2) .. (14.7,0);
	
	\vspace{3mm}

	\end{tikzpicture}
	
	\caption{Step 1 in Stage II -- Maker claims 3 edges and closes his path to a cycle.  }
	\label{fig:rotation}
\end{figure}
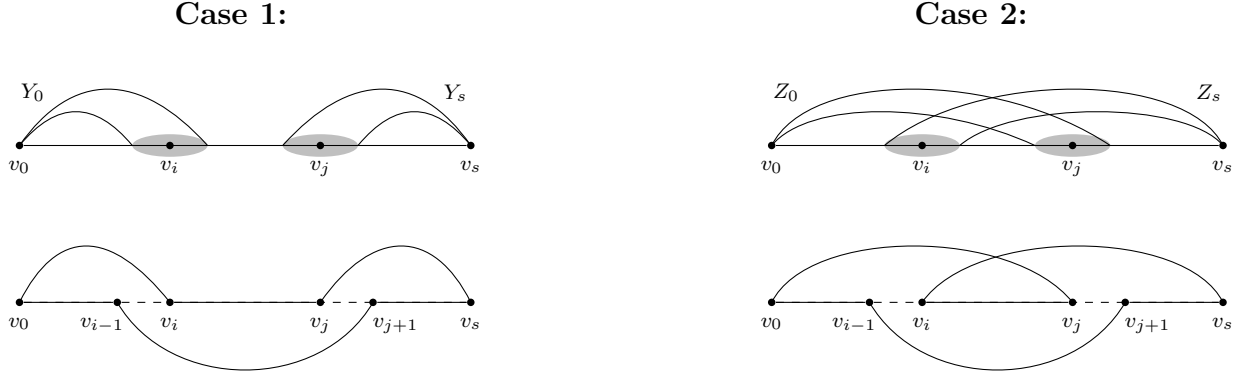

 	\item {\bf Step 2:} Denote  by $C$ the set of vertices in the cycle.
 	Maker claims some available edge $\{c,r\}$ , where $c\in C$ and $r\in
 	R$. Denote by $c'$ one of the neighbors of $c$ in the cycle.
 	Then Maker updates: $v_0\leftarrow c'$, $v_{s-1}\leftarrow c$, $s\leftarrow s+1$ and $v_s\leftarrow r$ and returns to Step 1. 
 
 	Following this strategy, in every iteration Maker increases his maximum degree by at most 2 and therefore at the end of this stage, $\Delta(M)\leq
 	2n^{1/4}+2$.
 
 \end{itemize}
 
 Now we prove that w.h.p.\ Maker can follow this strategy without forfeiting the
 game. In order to do so, first we need to show that during the
 game, the degree of Breaker in each vertex is not too large.

\begin{lemma}\label{lem:MinDegOfFInHam}
	Consider the $(1:b)$ Maker-Breaker game on $E(K_{n})$ where Breaker plays
	randomly and $b=(1-\varepsilon)\frac n2$. Assume that after $(1+o(1))n$ rounds 
	$\Delta(M)=O(n^{\delta})$ for a constant $0<\delta<1$. Then  w.h.p.\ $\Delta(B)\leq (1-\frac {\varepsilon}2)n$.
\end{lemma}

\begin{proof}
	Let $v\in V$ and denote by $d_B(v)$ the degree of $v$ in Breaker's graph after $(1+o(1))n$ rounds. Note that after $(1+o(1))n$ rounds Breaker claimed at most $(1-\frac {5}6\varepsilon)\binom n2$ edges. Also, since the maximal degree in Maker's graph is $O(n^{\delta})$ then by the end of the game Maker claimed at most $n^{\delta+1}$ edges. Thus, the number of edges Breaker claimed at each vertex is, essentially, asymptotically stochastically dominated by a hypergeometric random variable with appropriate parameters.  Then for $Z\sim HG\left(\binom n2-O(n^{\delta+1}),(1-\frac {5}6\varepsilon)\binom n2,n-1\right)$ and $\mu=\mathbb{E}(Z)$, using Lemma \ref{CheHyper} we have,
	\begin{align*}
	\Pr\left[d_B(v)> (1-\frac {\varepsilon}2)n\right]&\leq (1+o(1))\cdot\Pr\left[Z> (1-\frac {\varepsilon}2)n\right]\\
	&\leq (1+o(1))\cdot\Pr\left[Z> (1+\frac {\varepsilon}4)\mu\right]\\
	&\leq (1+o(1))e^{-\frac {\varepsilon^2}{48}\mu}\leq (1+o(1))e^{-\frac {\varepsilon^2}{48}(1-2\varepsilon)n}.
	\end{align*}
	Therefore, by union bound, 
	\begin{align*}
	\Pr\left[\exists v\in V,~s.t.~d_B(v)> (1-\frac {\varepsilon}2)n\right]\leq n\cdot (1+o(1))e^{-\frac {\varepsilon^2}{48}(1-2\varepsilon)n}=o(1).
	\end{align*}
\end{proof}

The next lemma shows that during the game, w.h.p.\ between any two sets of polynomial size, there are many available edges.

 \begin{lemma}\label{lem:NumOfFreeEdgInHam}
 	Consider the $(1:b)$ Maker-Breaker game on $E(K_{n})$ where Breaker plays
 	randomly and $b=\Theta (n)$. If at some point during the game
 	$|E(B)|\leq (1-\frac {5}6\varepsilon)\binom n2$, then w.h.p.\ for every constants $0<\beta,\gamma<1$ and for every $X,Y\subseteq V$ such that $|X|= n^{\beta}$, $|Y|= n^{\gamma}$ and $e_M(X,Y)=o(|X|\cdot |Y|)$, the number of available elements in $E(X,Y)$ is at least $\frac {\varepsilon}3 |X|\cdot |Y|$. 
 \end{lemma}

 \begin{proof}
 	Let $X,Y\subseteq V$ such that $|X|= n^{\beta}$ and $|Y|= n^{\gamma}$, then for $Z\sim HG\left(N,(1-\frac {5}6\varepsilon)\binom n2,|X| |Y|\right)$ where $N=\binom{n}{2}-o(n^2)$ and for $\mu=\mathbb E[Z]$,
 	
 	\begin{eqnarray}
 	\Pr\left[e_B(X,Y)>(1-\frac {\varepsilon}3) |X|\cdot |Y|\right]&\leq& (1+o(1))\Pr\left[Z>(1-\frac {\varepsilon}3)|X|\cdot|Y|\right]\nonumber\\
 	&\leq& (1+o(1)\Pr\left[Z>(1+\frac{\varepsilon}{4})\mu\right] \nonumber\\
 	&\leq& (1+o(1))e^{\frac{-\varepsilon^2}{48}\mu}
 	\leq (1+o(1)) e^{\frac{-\varepsilon^2}{48}(1-2\varepsilon)|X|\cdot|Y|}.\nonumber
 	\end{eqnarray}
 	
 	
 	Therefore, by the union bound 
 	
 	\begin{eqnarray}
 	&&\Pr\left[\exists X,Y,\ |X|= n^{\beta}, |Y|= n^{\gamma},\ s.t.\ e_B(X,Y)>(1-\frac {\varepsilon}3) |X|\cdot |Y| \right]\nonumber\\
 	&\leq& \binom{n}{|X|}\binom{n}{|Y|}\cdot(1+o(1))\cdot\exp\left[-{\frac{\varepsilon^2}{48}(1-2\varepsilon)n^{\beta+\gamma}}\right]\nonumber\\
 	&\leq& (1+o(1))\left( {en^{1-\beta}}\right)^{n^{\beta}}\left( {en^{1-\gamma}}\right)^{n^{\gamma}}\cdot\exp\left[-{\frac{\varepsilon^2}{48}(1-2\varepsilon)n^{\beta+\gamma}}\right]\nonumber\\
 	&=& o(1).\nonumber
 	\end{eqnarray}
 \end{proof}

 We now prove that w.h.p.\ Maker can follow his strategy at any point during the game.
 
 {\bf Stage I:}
 \begin{itemize}
 	
 	\item Step 1: Maker chooses one available edge $\{v_0,v_1\}$ and updates $P=\{v_0,v_1\}$. From Lemma \ref{lem:MinDegOfFInHam} we know that the number of available edges incident to $v_1$ is at least $\frac {\varepsilon}2n$. Therefore, in this case clearly $|R_{v_1}|\geq n^{1/5}$. 
 	
 	\item Step $k$: Let $R$ be the set of vertices that are not in the path $P$. Then $n^{1/4}\leq |R|\leq n$. We now show that if $|R_v|\geq n^{1/6}$ and $|R|\geq n^{1/4}$, then there are at least $n^{1/8}$ vertices $u\in R_v$ such that $|R_u|\geq n^{1/5}$. Let $T_v\subset R_v$ be the set of vertices such that for every $u\in T_v$, $|R_u|\geq n^{1/5}$. If $|T_v|<n^{1/8}$, then $e_F(R_v\setminus T_v,R)=\sum_{x\in R_v\setminus T_v}e_F(\{x\},R)\leq |R_v|\cdot n^{1/5}$. But according to Lemma \ref{lem:NumOfFreeEdgInHam}, w.h.p. $e_F(R_v\setminus T_v,R)\geq \frac {\varepsilon}3|R_v\setminus T_v|\cdot|R|\geq \frac {\varepsilon}4|R_v|\cdot n^{1/4}$. This is clearly a contradiction. Then, Maker chooses one vertex $u\in T_v$ and claims $\{v,u\}$.
 
 	In the next round, Breaker claims $b$ elements randomly. Let $X$ be the number of edges Breaker claimed from $R_u$ in the current round. Then $X\sim HG(\Theta(n^2),b,|R_u|)$ and  $\Pr(X>\varepsilon n^{1/5})=o\left(\frac 1n\right)$ and therefore after Breaker's move, $|R_u|\geq n^{1/6}$. Maker repeats this procedure $n-n^{1/4}$ times, then by union bound we have that w.h.p.\ in every step, $|R_u|\geq n^{1/6}$.

  \end{itemize}
 
 {\bf Stage II:}
 \begin{itemize}
 	\item Step 1: From Lemma \ref{lem:MinDegOfFInHam} we have that w.h.p.\ during the game, each vertex has at least $\frac {\varepsilon}3 n$ available neighbors. Assume that all the endpoints of $X_0$ and $X_s$ are in the path $P=\{v_0,v_1,\dots,v_s\}$.
 	\begin{itemize}
 		\item Case 1: Let $T_0$ be the set of the predecessors of the vertices of $Y_0$ along the path (without $v_0$), and let  $T_s$ be the set of the predecessors of the vertices of $Y_s$ along the path (without $v_s$). Then w.h.p.\ $|T_0|,|T_s|\geq \frac {\varepsilon}3n-1$. Therefore we can use Lemma \ref{lem:NumOfFreeEdgInHam} (note that $E_M(T_0,T_1)=O(n^{1/4})$) and get that w.h.p.\ there exists an available edge $\{v_i,v_j\}$ in the set $E(T_0,T_s)$. In the next 3 turns of Maker, he aims to claim the edges $\{v_i,v_j\},\{v_{i+1},v_0\},\{v_s,v_{j-1}\}$. Then $C=\{v_0,v_{i+1},\dots,v_{j-1},v_s,v_{s-1},\dots,v_{j},v_i,v_{i-1},\dots,v_0\}$ would be a cycle of Maker. The probability for Breaker to claim each such edge is $\Theta\left(\frac b{n^2}\right)$ and therefore (using the union bound over $n^{1/4}$ iterations) w.h.p.\ Maker can achieve his goal. 
 		\item Case 2: Let $T_0$  be the set of the predecessors of the vertices of $Z_0$ along the path (without $v_s$), and let $T_s$  be the set of the predecessors of the vertices of $Z_s$ along the path (without $v_0$). Then, $|T_0|,|T_s|\geq \frac {\varepsilon}3n-1$. Therefore we can use Lemma \ref{lem:NumOfFreeEdgInHam} and get that there exists an available edge $\{v_i,v_j\}$ in the set $E(T_0,T_s)$. In the next 3 turns of Maker, he aims to claim the edges $\{v_i,v_j\},\{v_{i+1},v_0\},\{v_s,v_{j-1}\}$. Then $C=\{v_0,v_{i-1},\dots,v_{j+1},v_s,v_{s-1},\dots,v_{i},v_j,v_{j-1},\dots,v_0\}$ would be a cycle of Maker. The probability for Breaker to claim each such edge is $\Theta\left(\frac b{n^2}\right)$ and therefore (using the union bound over $n^{1/4}$ iterations) w.h.p.\ Maker can achieve his goal. 
 	\end{itemize}
 	\item Step 2: From Lemma \ref{lem:MinDegOfFInHam}, we have that w.h.p.\ for every $v\in R$, $d_F(v)\geq \frac {\varepsilon}3n$. Since $|R|=O(n^{1/4})$ it follows that for some (in fact for any) $v\in R$, w.h.p.\ there exists a vertex $u\in C$ such that $\{u,v\}$ is available. 
 \end{itemize}

 \subsection{Random Breaker perfect-matching game}
 
 In order to prove Theorem \ref{thm:RB-PM}, we will use Maker's strategy in the perfect-matching game played on $E(K_{n,n})$. Therefore, we now analyze the random-Breaker perfect-matching game where the board of the game is $E(K_{n,n})$. Since this game is bias monotone, we can assume that $b= (1-\varepsilon)n$. We prove the following theorem:

 \begin{theorem}\label{thm:RB-PM2}
 	Let $\varepsilon>0$, let $n$ be a sufficiently large integer and let
 	$b=(1-\varepsilon)  n$. Then Maker has a 
 	strategy which is w.h.p.\ a winning strategy for the $(1:b)$ random-Breaker perfect-matching game played on $E(K_{n,n})$ in $(1+o(1)) n$
 	rounds.
 \end{theorem}

 Note that Theorem \ref{thm:RB-PM} follows easily from Theorem \ref{thm:RB-PM2}. Indeed, at the beginning of the game Maker focuses on two disjoint subsets of  $V(K_n)$, each of size $\lfloor \frac n2\rfloor$, and pretends to play the perfect-matching game on $E(K_{\lfloor \frac n2\rfloor,\lfloor \frac n2\rfloor})$  (Maker plays only on the edges between the two sets and ignores all other edges). Then, in every round, w.h.p.\ at least $(1- \frac{3\varepsilon}2)\lfloor \frac n2\rfloor$ of Breaker's edges will be contained in one of the two sets and therefore his  bias in the simulated game  is actually at most $(1- \frac{\varepsilon}2)\lfloor \frac n2\rfloor$.

 {\bf Proof of Theorem \ref{thm:RB-PM2}}
 First we present a strategy for Maker and then prove that w.h.p.\ during the
 game he can follow this strategy.
 
 {\bf Strategy S$_{PM}$:} Maker's strategy is divided into three stages. Denote
 by $B$ (or $M$), Breaker's (or Maker's) graph at any point during
 the game. We also write $G=M\cup B$.
 
 {\bf Stage I:} In this stage Maker's goal is to claim a matching $P$ of size $n-n^{\alpha}$ where $0<\alpha<\frac 13$ is a constant to be chosen later. Denote by $R$ the set of vertices not touched by any edge of $M$ and denote $E[R]=\{e\in E(K_{n,n})~|~|e\cap R|=2\}$. In every turn, Maker chooses two vertices $x,y\in R$ and claims the edge $\{x,y\}$. Maker repeats this procedure $n-n^{\alpha}$ times and then moves to the next stage. Note that at any point in this stage $\Delta(M)\leq 1$. If at some point during this stage Maker cannot follow this strategy, then he forfeits the game. This stage takes $n-n^{\alpha}$ rounds. 
 
 {\bf Stage II:} In this stage, Maker ``fixes" his graph by replacing some of the edges in the subgraph $P$. Denote again by $R$ the set of vertices such that $R\cap V(M_1)=\emptyset$ where $M_1$ is Maker's graph from Stage I and let $P'\leftarrow P$. Let $T=\{v\in R~|~e_F(v,R)<\frac {\varepsilon}{4}n^{\alpha}\}$. For every $v\in T$,  Maker chooses an edge $e=\{u,w\}\in P$ such that the edge $\{v,u\}$ is available, $e_F(w,R)\geq\frac {\varepsilon}{8}n^{\alpha}$, and then claims the edge $e'=\{v,u\}$ and updates $P\leftarrow P\setminus \{e\}$, $P'\leftarrow(P'\cup \{e'\})\setminus \{e\}$ and $R\leftarrow(R\setminus \{v\})\cup \{w\}$. Maker repeats this procedure $|T|$ times.  If at some point during this stage Maker cannot do so, then he forfeits the game. Denote Maker's new graph by $M_2$. Observe that by the end of this stage, $P'$ is a  matching of size $n-n^{\alpha}$ and that  $|R|=2n^{\alpha}$. Furthermore,  $\Delta(M_2)\leq 3$. We will show later that this stage takes at most $\Theta(n^{\alpha})$ rounds and that every vertex $v\in R$ satisfies $e_F(v,R)\geq\frac {\varepsilon}{5}n^{\alpha}$ by the end of this stage.
 
 {\bf Stage III:} Let $R$ the set of vertices from Stage II. Note that $|R|=2n^{\alpha}$. In this stage, Maker plays only on the graph $G'=K_{n,n}[R]=(U_1\cup U_2,E')$, note that $|U_1|=|U_2|=n^{\alpha}$.  At the beginning of this stage, Maker finds a set of available edges $P_1$ such that $P_1$ is a perfect matching in $G'$. In each turn, Maker claims an unclaimed edge $e=\{u,v\}$ from $P_1$. Maker repeats this procedure $n^{\alpha}$  times. If at some point during this stage Maker cannot do so, then he forfeits the game. Observe that during this stage, the degree of the vertices that are not in $ V(G')=R$ stays the same as in Stage II, and the degree of the vertices in $G'$ increases by one. Therefore, by the end of this stage $\Delta(M)=O(1)$.
 
It is evident that $P'\cup P_1$ is a perfect matching and therefore following the suggested strategy, w.h.p.\ Maker wins the game in $n-n^\alpha+\Theta(n^{\alpha})+n^{\alpha}=n(1+o(1))$ rounds. It remains to prove that  w.h.p.\ Maker can follow this strategy without forfeiting the game. 
To this end, we first state two lemmas concerning Breaker's graph by the end of the game. The proofs of the lemmas are essentially identical to the proofs of Lemma \ref{lem:NumOfFreeEdgInHam}  and Lemma \ref{lem:MinDegOfFInHam}, respectively.

 \begin{lemma}\label{lem:NumOfFreeEdg}
 	Consider the $(1:b)$ Maker-Breaker game on $E(K_{n,n})$ (the bipartite complete graph with sides $V_1,V_2$), where Breaker plays
 	randomly and $b=\Theta (n)$. If at some point of the game
 	$|E(B)|\leq (1-\frac {5}6\varepsilon)n^2$ and $\Delta(M)=O(1)$, then w.h.p.\ for every constants $0<\beta,\gamma<1$ and for every $X_1\subseteq V_1$, $X_2\subseteq V_2$  such that $|X_1|= n^{\beta}$, $|X_2|= n^{\gamma}$, the number of available elements in $E(X_1,X_2)$ is at least $\frac {\varepsilon}3 |X_1|\cdot |X_2|$. 
 \end{lemma}

 \begin{corollary}
 	\label{lem:enoughEdgesInR}
 	Consider the $(1:b)$ Maker-Breaker game on $E(K_{n,n})$ where Breaker plays
 	randomly and $b=(1-\varepsilon) n$. If 
 	$\Delta(M)=O(1)$, then after $(1+o(1))n$ rounds w.h.p.\ between every two sets of vertices $X_i\subset V_i$ ($i\in \{1,2\}$) such
 	that $|X_i|\geq n^{\alpha}$, there exists an unclaimed edge.
 \end{corollary}
 
The next lemma shows that during the game, the degree in Breaker's graph is not too close to $n$.
 
 \begin{lemma}\label{lem:MinDegOfF}
 	Consider the $(1:b)$ Maker-Breaker game on $E(K_{n,n})$ where Breaker plays
 	randomly and $b=(1-\varepsilon) n$. If 
 	$\Delta(M)=O(1)$, then after $(1+o(1))n$ rounds, w.h.p.\ $\Delta(B)\leq (1-\frac {\varepsilon}2)n$.
 \end{lemma}
 

 Now we are ready to prove that w.h.p.\ Maker can follow his strategy in every stage.
 
 {\bf Stage I:} We need to prove that at any point during this stage Maker can find unclaimed edge $\{x,y\}$ such that $\{x,y\}\cap V(M)=\emptyset$. Indeed, at any point during this stage we have $|R|\geq 2n^{\alpha}$, and by Corollary \ref{lem:enoughEdgesInR} we know that there is an unclaimed edge in $R$. Therefore, Maker can claim the desired edge.
 
 {\bf Stage II:} First we show that the number of vertices $v\in R$ such that $e_F(v,R)<\frac {\varepsilon}8 n^{\alpha}$ is small.
 
Using Lemma~\ref{lem:NumOfFreeEdg} for $R$, w.h.p.\ the graph $F$, by the end of the game, satisfies the property from Claim \ref{claim:TisSmall}, and thus, w.h.p.\ $|T|\leq 2n^{\alpha/2}$. We write $T_0=V_0\cap T$ and $T_1=V_1\cap T$ where $V_0$ and $V_1$ are the parts of $K_{n,n}$. Then $|T_1|,|T_0|\leq n^{\alpha/2}$.
 
  In order to replace the \emph{bad} vertices in $R$ (that is, the vertices of $T$), note that for every vertex $v\in T_i$ the number of candidates to replace $v$ is large. Indeed, from Lemma~\ref{lem:MinDegOfF}, w.h.p.\ the minimum degree in $F$ is at least $\frac {\varepsilon}4n$, therefore w.h.p.\ for every $v\in T_i$, $e_F(v,V_{1-i}\setminus (R\cup T))\geq \frac {\varepsilon}4n-2n^\alpha-2n^{\alpha/2}\geq \frac {\varepsilon}5n$. For every $v\in T_i$ denote  $W_v=\{w\in V_i\setminus (T\cup R)~|~\exists u,~\{w,u\}\in P,~\{v,u\}\in F\}$, so $|W_v|\geq \frac {\varepsilon}5n>n^{\alpha}$. Since  w.h.p.\ the graph $F$ satisfies the property from Claim~\ref{claim:TisSmall}, we have that w.h.p.\ there exists $w\in W_v$ such that $e(w,U_{1-i})\geq \frac {\varepsilon}4 n^{\alpha}$. 
  Therefore, for every $v\in T_0$, Maker can choose a vertex $u\in V_1\setminus (R\cup T)$ such that $\{v,u\}$ is available and  $e_F(w,R)\geq \frac {\varepsilon}4 n^{\alpha}$ where $w$ is the vertex such that  $\{u,w\}\in P$. Maker repeats this procedure for every vertex in $v\in T_0$ and then move to the vertices of $T_1$. Thus, we can replace the vertices of $T_0$ in such a way that for every $v\in U_0$, $e_F(v,U_1)\geq \frac {\varepsilon}4n^{\alpha}$. Finally, Maker does the same for the vertices in $T_1$ and ensure that for every $v\in U_1$, $e_F(v,U_0)\geq \frac {\varepsilon}4n^{\alpha}$. This time, during the procedure, if a vertex $v\in T_1$ was replaced by $w$, then the degree of his neighbors in $U_0$ could become smaller than $\frac {\varepsilon}4n^{\alpha}$. But $|T_1|\leq n^{\alpha/2}$ and therefore after replacing the vertices of $T_1$ we still have that for every $v\in U_0$, $e_F(v,U_1)\geq\frac {\varepsilon}4n^{\alpha}-n^{\alpha/2}>\frac{\varepsilon}5n^{\alpha}$. 
 Note that in every step Maker looks only at the neighbors of $v\in T$ that are not in the current $R$ and are not part of $T$. Therefore, by the end of this stage $\Delta(M_2)\leq 3$. This completes the proof for Stage II. 

{\bf Stage III:} First we show that after Stage II, the graph $F[R]$ contains a perfect matching $P_1$.

By the construction in Stage II and by Lemma~\ref{lem:NumOfFreeEdg}, we have that the graph $G'$ w.h.p.\ satisfies properties 1 and 2 of Claim~\ref{claim:PMinR} with $|A_i|=n^{\alpha}$. Therefore, w.h.p.\ $G'$ contains a perfect matching $P_1$.

Finally, we show that in the next $n^{\alpha}$ rounds, w.h.p.\ Breaker will not claim any edge from $G'$.
Recall that by the end of the game there are at least $ \frac 34\varepsilon n^2$ available edges in the graph. Thus,	the probability for Breaker to claim an edge from $G'$ in the next round is at most $\frac {4n^{2\alpha}}{3\varepsilon n^2}\cdot b=(1-\varepsilon)\frac {4n^{2\alpha +1}}{3\varepsilon n^2}$. Therefore, the probability for Breaker to claim an edge from $G'$ in the next $n^\alpha$ rounds is at most $n^{\alpha}\cdot (1-\varepsilon)\frac {4n^{2\alpha +1}}{3\varepsilon n^2}=(1-\varepsilon)\frac {4n^{3\alpha +1}}{3\varepsilon n^2}=o(1)$ for $0<\alpha < \frac 13$.  We thus choose $\alpha$ to be (for instance) $\frac 14$.

All in all, during stage III w.h.p.\ Breaker will not claim edges from $G'$. Furthermore, from Claim~\ref{claim:PMinR}, w.h.p.\ Maker can find a perfect matching $P_1$ in $G'$ and thus claim all the edges of $P_1$ in the next $n^{\alpha}$ rounds. 

This completes the proof of Theorem~\ref{thm:RB-PM}.
 
 \subsection{Random Breaker $k$-connectivity game}

 In order to prove Theorem~\ref{thm:RB-C^k} we use the Random-Breaker Hamiltonicity game (Theorem~{\ref{thm:RBHam}}) and the Random-Breaker perfect-matching game (Theorem~\ref{thm:RB-PM2}).  Since this game is also bias monotone, we can assume that $b= (1-\varepsilon)\frac nk$. 
 
 First we present a strategy for Maker and then we prove that w.h.p.\ Maker can follow this strategy.
 
 {\bf Strategy $S_{k}$:} Maker (arbitrarily) partitions the vertices of $K_n$ into $k$ disjoint sets $V_1,\dots, V_{k-1},U$ where each $V_i$ has size $\lfloor{\frac {n}{k-1}}\rfloor$ and $U=V(K_n)\setminus \left(\bigcup V_i\right)$ (that is, $|U|<k-1$). For every $1\leq i\leq k-1$ let $G_i=K_n[V_i]$ and for every $1\leq i<j\leq k-1$ let $B_{ij}$ be the bipartite graph with $V(B_{ij})=V_i\cup V_j$ and $E(B_{ij})=\{\{u,v\}~|~v\in V_i,~u\in V_j\}$. 
 
  {\bf Stage I:} Maker plays the perfect-matching game on every $B_{ij}$ according to $S_{PM}$ with $\alpha=\frac 14$.  Maker follows $S_{PM}$ on $E(B_{12})$ to create a perfect matching in $B_{12}$, then on $E(B_{13})$, etc. In total, Maker plays (sequentially) $\binom {k-1}2$ games on $\binom {k-1}2$ boards. If at some point Maker cannot follow $S_{PM}$ then he forfeits the game. Note that this stage takes $(1+o(1))\lfloor{\frac {n}{k-1}}\rfloor\binom {k-1}2$ rounds and by the end of this stage Maker's graph contains a perfect matching between any two sets $V_i,V_j$. Also, according to $S_{PM}$, $\Delta(M)=O(1)$.
  
   {\bf Stage II:} Maker plays the Hamiltonicity game on every $G_i$ according to $S_{Ham}$.  Maker follows $S_{Ham}$ on $G_1$ to create a Hamilton cycle in $G_1$, then on $G_2$, etc. In total, Maker plays $k-1$ games on $k-1$ boards. If at some point Maker cannot follow $S_{Ham}$ then he forfeits the game. Note that this stage takes $(1+o(1))\lfloor{\frac {n}{k-1}}\rfloor(k-1)$ rounds and by the end of this stage Maker has $k-1$ disjoint cycles each of length $\lfloor{\frac {n}{k-1}}\rfloor$. Also, according to $S_{Ham}$ and $S_{PM}$, $\Delta(M)=O(n^{1/4})$.
   
  {\bf Stage III:} For every $u\in U$ Maker claims $k$ distinct edges $\{u,w\}$ such that $w\in V(K_n)\setminus U$. This stage takes at most $k^2$ rounds.
  
  Observe that if Maker successfully follows the strategy, his graph (on $\bigcup V_i$) consists of $k-1$ cycles of size $\lfloor\frac n{k-1}\rfloor$ and between any two  cycles there is a perfect matching. This graph is easily seen to be $k$-connected (see e.g., \cite{CFKL}, Lemma 2.7). Adding $U$ to this graph and demanding that for every $u\in U$, $e(\{u\},\cup V_i)=k$, we conclude that Maker's graph is indeed $k$-connected. In total, Maker plays at most $\left\lfloor{\frac {n}{k-1}}\right\rfloor (k-1)+\left\lfloor{\frac {n}{k-1}}\right\rfloor\binom {k-1}2+k^2+o(n)=(1+o(1))\frac {nk}2$ turns.
  
  We show now that w.h.p.\ Maker can follow strategy $S_k$. 
  
  {\bf Stage I:}  First we show that the conditions of Lemma~\ref{lem:NumOfFreeEdg}  holds. Indeed, observe that by the end of the game,  Breaker has at most $(1-\frac 56\varepsilon)\binom n2$ edges in the graph. Therefore, it follows from Lemma~\ref{lem:NumOfFreeEdgInHam} that w.h.p.\ by the end of the game $|E_B(B_{ij})|=e_B(V_i,V_j)\leq (1-\frac {\varepsilon}3)\left(\left\lfloor{\frac {n}{k-1}}\right\rfloor\right)^2$ for every $1\leq i<j\leq k-1$. Thus, the statement in Lemma~\ref{lem:NumOfFreeEdg} also holds for every graph $B_{ij}$ where $\frac {\varepsilon}3 \leftarrow \frac 56\varepsilon$. Next we show, analogously to Lemma~\ref{lem:MinDegOfF}, that during the game, the degree in Breaker's graph in each $B_{ij}$ is not too large.

   \begin{lemma}\label{lem:DegOfBreakerInB_ij}
   	Consider the $(1:b)$ Maker-Breaker game on $E(K_{n})$ where Breaker plays
   	randomly and $b=(1-\varepsilon) \frac nk$. If after $(1+o(1))\frac {kn}2$ turns of the game
   	$\Delta(M)=O(1)$, then w.h.p.\ $\Delta(B\cap B_{ij})\leq (1-\frac {\varepsilon}2)\lfloor{\frac {n}{k-1}}\rfloor$.
   \end{lemma}
   
   \begin{proof}
   	Let $v\in V$ and denote by $d_B(v)$ the degree of $v$ in $B[B_{ij}]$ by the end of the game. Then for $Z\sim HG\left(\binom n2-O(n),(1-\varepsilon)\binom n2,\left\lfloor{\frac {n}{k-1}}\right\rfloor\right)$ and $\mu=\mathbb{E}(Z)$, using Lemma \ref{CheHyper} we have,
   	\begin{align*}
   	\Pr\left[d_B(v)> (1-\frac {\varepsilon}2)\left\lfloor{\frac {n}{k-1}}\right\rfloor\right]&\leq (1+o(1))\cdot\Pr\left[Z> (1-\frac {\varepsilon}2)\left\lfloor{\frac {n}{k-1}}\right\rfloor\right]\\
   	&\leq (1+o(1))\cdot\Pr\left[Z> (1+\frac {\varepsilon}4)\mu\right]\\
   	&\leq (1+o(1))e^{-\frac {\varepsilon^2}{48}\mu}\leq (1+o(1))e^{-\frac {\varepsilon^2}{48k}(1-2\varepsilon)n}.
   	\end{align*}
   	Therefore, by the union bound, 
   	\begin{align*}
   	\Pr\left[\exists v\in B_{ij},~s.t.~d_B(v)> (1-\frac {\varepsilon}2)\left\lfloor{\frac {n}{k-1}}\right\rfloor\right]\leq 2\left\lfloor{\frac {n}{k-1}}\right\rfloor\cdot (1+o(1))e^{-\frac {\varepsilon^2}{48k}(1-2\varepsilon)n}.
   	\end{align*}
   	Finally, 
   	\begin{align*}
   &\Pr\left[\exists B_{ij},~\exists v\in B_{ij},~s.t.~d_B(v)> (1-\frac {\varepsilon}2)\left\lfloor{\frac {n}{k-1}}\right\rfloor\right]\\
   \leq& \binom {k-1}2\cdot2\left\lfloor{\frac {n}{k-1}}\right\rfloor\cdot (1+o(1))e^{-\frac {\varepsilon^2}{48k}(1-2\varepsilon)n}=o(1).
   	\end{align*}
    \end{proof}

   All in all, according to the proof of Theorem~\ref{thm:RB-PM2} (and using the union bound on $\binom {k-1}2$ events), Maker can build w.h.p.\ a perfect matching in each $B_{ij}$ in $(1+o(1))\lfloor{\frac {n}{k-1}}\rfloor$ rounds. Therefore, to build $\binom {k-1}2$ perfect matchings Maker needs $(1+o(1))\lfloor{\frac {n}{k-1}}\rfloor\binom {k-1}2$ rounds.
   
    {\bf Stage II:}  In order to follow the strategy in this stage, we need to show that similarly to the proof of Theorem \ref{thm:RBHam}, also here there are many  available edges between any two subsets of vertices of polynomial size. For this we use  Lemma \ref{lem:NumOfFreeEdgInHam}. Note that it follows from the strategy and the duration of the game that yet again w.h.p.\ $|E(M)|=o(n^2)$ and  $|E(B)|\leq (1-\frac 56 \varepsilon)\binom n2$ and thus the statement of Lemma \ref{lem:NumOfFreeEdgInHam} holds here as well.
    	
%
%
%
%


   Next we show, similarly to Lemma \ref{lem:MinDegOfFInHam}, that during the game, the degree in Breaker's graph in each $G_i$ is not too large.

    \begin{lemma}\label{lem:DegOfBreakerInG_i}
    	Consider the $(1:b)$ Maker-Breaker game on $E(K_{n})$ where Breaker plays
    	randomly and $b=(1-\varepsilon) \frac nk$. If 
    	$\Delta(M)=O(n^{\delta})$ (where $0<\delta<1$ is a constant), then w.h.p.\ $\Delta(B\cap G_{i})\leq (1-\frac {\varepsilon}2)\lfloor{\frac {n}{k-1}}\rfloor$.
    \end{lemma}
    
    \begin{proof}
    	Let $v\in V$ and denote by $d_B(v)$ the degree of $v$ in $B[G_i]$ by the end of the game. Then for $Z\sim HG\left(\binom n2-n^{\delta+1},(1-\varepsilon)\binom n2,\lfloor{\frac {n}{k-1}}-1\rfloor\right)$ and $\mu=\mathbb{E}(Z)$, using Lemma \ref{CheHyper} we have,
    	\begin{align*}
    	\Pr\left[d_B(v)> (1-\frac {\varepsilon}2)\left\lfloor{\frac {n}{k-1}}\right\rfloor\right]&\leq (1+o(1))\cdot\Pr\left[Z> (1-\frac {\varepsilon}2)\left\lfloor{\frac {n}{k-1}}\right\rfloor\right]\\
    	&\leq (1+o(1))\cdot\Pr\left[Z> (1+\frac {\varepsilon}4)\mu\right]\\
    	&\leq (1+o(1))e^{-\frac {\varepsilon^2}{48}\mu}\leq (1+o(1))e^{-\frac {\varepsilon^2}{48k}(1-2\varepsilon)n}.
    	\end{align*}
    	Therefore, by union bound, 
    	\begin{align*}
    	\Pr\left[\exists v\in G_{i},~s.t.~d_B(v)> (1-\frac {\varepsilon}2)\left\lfloor{\frac {n}{k-1}}\right\rfloor\right]\leq \left\lfloor{\frac {n}{k-1}}\right\rfloor\cdot (1+o(1))e^{-\frac {\varepsilon^2}{48k}(1-2\varepsilon)n}.
    	\end{align*}
    	Finally, 
    	\[\Pr\left[\exists G_i,~\exists v\in G_i,~s.t.~d_B(v)> (1-\frac {\varepsilon}2)\left\lfloor{\frac {n}{k-1}}\right\rfloor\right]\leq (k-1)\cdot\left\lfloor{\frac {n}{k-1}}\right\rfloor\cdot (1+o(1))e^{-\frac {\varepsilon^2}{48k}(1-2\varepsilon)n}=o(1).\]
    \end{proof}
   
   All in all, according to the proof of Theorem~\ref{thm:RBHam} (and using the union bound on $k-1$ events), w.h.p.\ Maker can build a Hamilton cycle in each $G_i$ in $(1+o(1))\lfloor{\frac {n}{k-1}}\rfloor$ rounds. Therefore, to build ${k-1}$ cycles, Maker needs $(1+o(1))\lfloor{\frac {n}{k-1}}\rfloor (k-1)$ rounds.
   
   {\bf Stage III:} From Lemma~\ref{lem:MinDegOfF} we know that w.h.p.\ for every $u\in U$ we have $d_B(u)\leq (1-\frac {\varepsilon}2)n$. But $|U|<k$ and therefore w.h.p.\ by the end of the game $e_F(\{u\},V\setminus U)>k$ and Maker can  claim w.h.p.\ $k$ neighbors of $u$ in $k$ rounds. 
  
  Taking it all together, w.h.p.\ Maker can follow the suggested strategy and build the desired graph. This completes the proof of Theorem \ref{thm:RB-C^k}.

\section{Random Maker games}
In this section we consider the random-player setting where Maker plays randomly and claims in every round $m$ elements among all available elements, uniformly at random. Here we prove Theorems \ref{thm:RMIsoVertex}, \ref{thm:RMHam} and \ref{thm:RMCon}.

\subsection{Random Maker positive minimum degree game -- Breaker's side}

In this section we prove Theorem \ref{thm:RMIsoVertex}.   We show that if $m=O(\ln\ln n)$ then w.h.p.\ Breaker can isolate a vertex in Maker's graph and thus typically win
every game whose winning sets consist of spanning subgraphs with a positive minimum degree. 

\begin{proof} Let $\varepsilon>0$ be small enough constant and let $m\leq(\frac 12-\varepsilon)\ln\ln n$. Since claiming more edges can only help Maker, we can assume $m=(\frac 12-\varepsilon)\ln\ln n$.  We show that Breaker has a strategy such that following it w.h.p.\ he can claim all $n-1$ edges incident to one vertex within
	$\frac {cn\ln n}m$ rounds of the game ($c$ is a constant to be
	chosen later). Observe that after $\frac {cn\ln n}m$ rounds, there
	are only $\frac {cn\ln n}m(m+1)=(1+o(1))cn\ln n$ claimed elements on the
	board of the game. Therefore, at any point during the game, there
	are at least $\binom n2-\frac {cn\ln n}m(m+1)=(1-o(1))\frac {n^2}2$
	available elements on the board. We say that a vertex $v$ is \emph{free} if $E_M(\{v\},N(v)) =\emptyset$.  Breaker chooses a free
	vertex $v$ and tries to claim all the edges incident  to $v$. If Maker claimed some edge $\{v,w\}$ before Breaker was able to claim all edges incident to $v$, then in the next round Breaker chooses another free vertex and repeats the procedure.  If at some point during the first $\frac {cn\ln n}{m}$ rounds there is no free vertex in the graph, then Breaker forfeits the game.  Every time Breaker moves to a new free vertex, we say that Breaker started a new \emph{attempt}. Now we need to show that w.h.p.\ if Breaker follows his strategy then he will succeed in one of the attempts. 
	
	First we show that at any point during the first $\frac {cn\ln n}m$
	rounds of the game, there exists a free vertex in the graph.
	
	\begin{claim}\label{claim:freeVertex}
		There exists a constant $c>0$ such that after $\frac {cn\ln n}m$ rounds of the game, w.h.p.\ there exists a vertex
		$v$ which is isolated in Maker's graph.
	\end{claim}
	
	\begin{proof}
		In the first $\frac {cn\ln n}m$ rounds, Maker claims $cn\ln n$ edges.
		The proof is similar to the Coupon Collector problem with minor
		changes. 
%
%
Assume
that every time that Maker touches a free vertex, he actually
touches two free vertices (otherwise it is only in favor of Breaker). 
		Let
		$\tau_i$ be the number of turns until Maker
		touches $2i$ different vertices and let $\tau$ be the number of
		turns until Maker touches all the vertices. Then
		$\tau=\tau_1+(\tau_2-\tau_1)+(\tau_3-\tau_2)+\dots+(\tau_{\frac
			n2}-\tau_{\frac n2-1})$. Also, for every $i$, $\tau_{i+1}-\tau_i$ is stochastically dominated from above by $A_i\sim Geo(\frac{2}{1-o(1)}\cdot\frac {n-2i}{n})$ and from below by $B_i\sim Geo(\frac {(n-2i)n-cn\ln n}
		{n^2})$ (the number of
		unclaimed edges from the $n-2i$ free vertices  divided by the number of available edges in the graph).  Furthermore, since $B_i$ and $A_i$ are all
		independent random variables, we have that
		$\mathbb{E}(\tau)\geq Cn\ln n$ (for some constant $C>0$) and
		$Var(\tau)=\Theta(n^2)$. Then by Chebyshev we have that
		$$\Pr\left(|\tau-Cn\ln n|>dn\ln \ln n\right)\leq \Theta\left(\frac{n^2}{d^2n^2\ln^2 \ln n}\right)=o(1).
		$$ Therefore, for a constant $c$ satisfying $cn\ln n<Cn\ln n-dn\ln\ln
		n$ we have that $\Pr\left(\tau<cn\ln n\right)=o(1).$
		%
		%
		%
	\end{proof}
	
	\medskip
	
	Next, we show that the probability for Breaker to claim
	all edges from one vertex in one attempt is not too small.
	
	\begin{claim}
		The probability for Breaker to be able to claim in one attempt all edges incident to some {\bf
			free} vertex $v$ is at least $\frac {1}{\ln^{1-4\varepsilon} n}$.
	\end{claim}
	
	\begin{proof}
		For some vertex $v$, the probability for Maker to touch this vertex
		at some turn is at most $\frac {n-1}{\binom n2-\frac {cn\ln
				n}m(m+1)}=\frac 2n (1+o(1))$. Therefore, the probability for Breaker
		to claim all edges from some vertex is at least $(1-\frac 2n
		(1+o(1)))^{m(n-1)}$ (we want Maker to "fail`` at least $m(n-1)$ times). Note
		that,
		
		\begin{align*}
		\left(1-\frac 2n
		(1+o(1))\right)^{m(n-1)}=&\left(1-\frac 2n (1+o(1))\right)^{(\frac 12-\varepsilon)(n-1)\ln \ln
			n}
		\geq e^{-(1-4\varepsilon)\ln \ln
			n},
		\end{align*}
		this completes the proof.
	\end{proof}

	During the game, Breaker has at least $\frac
	{\frac {cn\ln n}m}{n}=\frac {c\ln n}{m}$ attempts. Let $X$ be the
	number of successful attempts, then,
	$$\Pr\left[X<1\right]\leq \Pr\left[\Bin\left(\frac {c\ln n}{m},\frac {1}{\ln^{1-4\varepsilon} n}\right)=0\right]=\left(1-\frac {1}{\ln^{1-4\varepsilon} n}\right)^{\frac {c\ln n}{m}}=o(1).$$

	Thus w.h.p.\ Breaker succeeds at least once and is able to claim all edges from some vertex $v$ in the first $\frac {cn\ln n}m$ rounds. This completes the proof.
\end{proof}

\subsection{Random Maker Hamiltonicity game}

In this section we prove Theorem~\ref{thm:RMHam}. One way to show that Maker's graph contains w.h.p.\ a Hamilton cycle, is to prove that Maker's graph is a good (connected) expander. In order to do so, we use the well-known Box game, firstly introduced by Chv\'atal and Erd\H{o}s in \cite{CE}. 

Let $n$ and $s$ be integers. In the $(a:b)$-$Box(n\times s)$ game there are $n$ boxes, each with $s$ elements. In every round, BoxMaker claims $a$ elements from the boxes, and BoxBreaker claims $b$ elements from the boxes. BoxBreaker's goal is to claim at least one element from each box by the end of the game. After we analyze the game Box in our setting (and a variant called the $d$-Box game), we show that Maker's graph is indeed w.h.p.\ an expander and then we show that w.h.p.\ it contains a Hamilton cycle.

\subsubsection {Building an expander}

\begin{lemma}\label{lemma:box2}
	Let $n$ be an integer, let $d$ and $\alpha$ be constants and let $A>9d/\alpha$. Then for $b=A\ln \ln n$ w.h.p.\ random-BoxBreaker wins the
	$(1:b)$-random-BoxBreaker $Box(dn\times s)$ game assuming $s=\alpha n$.
\end{lemma}

\begin{proof} {\bf of Lemma \ref{lemma:box2}}
	We show that w.h.p.\ BoxBreaker wins within $\Theta(\frac {n\ln n}{b})$
	rounds. Specifically, we argue that after $\frac {4dn\ln n}{b}$ rounds, the elements claimed by BoxBreaker w.h.p.\ satisfy two properties that allow BoxBreaker to  win the game.
	
	{\bf Property 1:} First we show that in the first  $\frac {4dn \ln n}{b}$ rounds,
	w.h.p.\ every box that BoxBreaker has not touched, has at least $
	\frac s3$ available elements. Let $A_i$ be the event
	"the $i^{th}$ box to reach $\frac s3$ elements of BoxMaker was not touched by BoxBreaker
	in the next $\frac s3$ rounds``. Now, the probability for BoxBreaker to
	touch some box $i$ with his next element is at least $\frac {s/3}{dns}= \frac
	{1}{3dn}$, Therefore,
	$$\Pr\left(A_i\right)\leq \left(1-\frac {1}{3dn}\right)^{bs/3}\leq e^{-\frac {1}{9dn}bs }=
	e^{-\frac {\alpha}{9d} A\ln \ln n}=\left(\frac 1{\ln n}\right)^{\frac {\alpha}{9d}A}.$$
	Note that the number of events $A_i$ that can occur is at most
	$\frac {4dn\ln n/b}{s/3}=\frac {12d\ln n}{\alpha b}$. Thus, for $A>9d/\alpha$ we
	have that
	$$\Pr\left(\exists i\in[n],\ s.t.\ A_i\right)\leq \frac {12d\ln n}{\alpha b}\cdot \Pr\left(A_i\right)
	\leq \frac {12d\ln n}{\alpha A\ln\ln n}\cdot \frac 1{(\ln n)^{\frac
			{\alpha}{9d}A}}=o(1).$$
	
	{\bf Property 2:} Next we show that w.h.p. BoxBreaker touches every box that had at
	least $\frac s3$ available elements in the first $\frac {4dn\ln n}{b}$
	rounds. We say that a box $i$ is \emph{free} if all the elements in $i$ are either available or belong to BoxMaker.  The probability for BoxBreaker to
	touch the box $i$ with at least $\frac s3$ available elements in a given turn is at least $\frac 1{3dn}$. Therefore, for a given box, the probability that after the first $\frac {4dn\ln n}{b}$ rounds the box has at least $\frac s3$ available elements and was not  touched by BoxMaker is at most 	 $ \left(1-\frac 1{3dn}\right)^{\frac {4dn\ln n}{b}}= o\left(\frac 1{dn}\right)$. Using the union bound over all boxes we have that the probability that after the first $\frac {4dn\ln n}{b}$ rounds there is a box that was not touched by BoxBreaker is $1-o(1)$.
	
	
	All in all, the set of elements belong to BoxBreaker w.h.p.\ satisfies properties 1+2 and thus after $\frac {4dn\ln n}{b}$ rounds (that is $4dn\ln n$
	turns of BoxBreaker), w.h.p.\ BoxBreaker claimed at least one element
	from each box.
\end{proof}

\bigbreak

The following corollary is obtained from Lemma \ref{lemma:box2}.
In this corollary, we study a version of the Box game, the $d$-Box game, where $d$ is a positive integer. This game will be used later for proving Theorem
\ref{thm:RMHam} and Lemma \ref{lemma:buildExpander}. In the game
$d$-$Box(n\times s)$ there are $n$ boxes, each with $s$ elements, and two players,
$d$-Maker and $d$-Breaker. In the $(m:1)$-random-Maker $d$-Box game, $d$-Maker plays randomly and
claims in each round exactly $m$ previously unclaimed elements while $d$-Breaker claims exactly one element. The goal of $d$-Maker is to
claim at least $d$ elements in each box.

\begin{corollary}\label{claim:MinDeg}
	Let $\alpha, d,n>0$ be integers, let $A>9d^2/\alpha$ be a constant and let $m=A\ln \ln n$.  Then w.h.p.\ $d$-Maker wins the $(m:1)$-random-Maker
	$d$-$Box(n\times s)$ game where 
	$s=\alpha n$. Furthermore, in every box there were at least $\frac s2$ available elements  when $d$-Maker claimed $\frac d2$ elements from it. 
\end{corollary}

\begin{proof}
	At the beginning of the game, $d$-Maker partitions each of the $n$
	boxes into $d$ sub-boxes, each of size $\frac {s}{d}$.  Then, for $s=\alpha n$ we have $dn$ boxes each of size $\frac {\alpha n}{d}$. From Lemma \ref{lemma:box2}, playing the random-BoxBreaker $Box(dn\times \frac {\alpha n}{d})$ game, BoxBreaker w.h.p.\ wins when $b\geq A\ln\ln dn$ where $A>9d^2/\alpha$. Therefore for $m\geq A\ln\ln n$ and $A>9d^2/\alpha$ $d$-Maker wins w.h.p.\ the $d$-Box game. Moreover, he typically does so within $\Theta(\frac {n\ln n}m)$ rounds of the game.

	For the second part, we showed that w.h.p.\ $d$-Maker is the winner of the game against any strategy of $d$-Breaker. Now, assume by contradiction that at some point of the game there exists a box $i$ with at least $\frac s2-\frac d2+1$ elements of $d$-Breaker and at most $\frac d2-1$ elements of $d$-Maker. Then, by looking at $d$-Maker's partition  of box $i$ into $d$ sub-boxes, $d$-Breaker could have claimed all the elements (but one) from $\frac d2-1$ sub-boxes and all the $\frac sd$ elements from another sub-box, thus winning the game, which is clearly a contradiction. 
\end{proof}

\medskip
The next step in the proof of Theorem \ref{thm:RMHam} (and also of Theorem \ref{thm:RMCon}) is to show that w.h.p.\ Maker's graph is an expander. For this we need the following lemma.

\begin{lemma}\label{lemma:buildExpander}
	For every positive integer $k$ there exist constants
	$\delta>0$ and $C_1>0$ for which the
	following holds. Suppose that $m'\geq {C_1\ln\ln n}$. Then in the $(m':1)$-random-Maker game played on the edge set of $K_n$, w.h.p.\ after $O(\frac {n\ln n}{m'})$ rounds, Maker's graph is an $(R,2k)$-expander,
	where $R=\delta n$.
\end{lemma}

\begin{proof}
	Let $d=16k$ and let $C_1>20d^2$.
	At the beginning of the game, Maker assigns edges of $K_n$ to
	vertices so that each vertex gets about $\frac n2$ edges
	incident to it. To do so, let $D_n$ be any tournament on $n$
	vertices such that for every vertex $v\in V$, $|N^+(v)|=|N^-(v)|\pm
	1$ if $n$ is even and $|N^+(v)|=|N^-(v)|$ if $n$ is odd. For each
	vertex $v\in V(D_n)$, define $A_v=E(v,N^+(v))$. Note that
	for every $v\in V(K_n)$ we have that
	$|A_v|=\lfloor\frac{n-1}2\rfloor$ or $|A_v|=\lceil\frac{n-1}2\rceil$
	and that all the $A_v$'s are pairwise disjoint. For the sake of simplicity, during the proof we sometimes omit floor
	and ceiling signs whenever these are not crucial.
	
	Now, note that if a graph $G$ is an $(R,2k)$-expander, then $G\cup\{e\}$ is
	also an $(R,2k)$-expander for every  edge $e\in E(K_n)$. Therefore,
	claiming extra edges can not harm Maker in his goal of creating an
	expander and we can assume $m'= {C_1n\ln n}$ (if $m'$ is larger
	then it is only in favor of Maker).
	
	Our goal is to show that w.h.p.\ Maker's graph, after $\Theta(n\ln n)$ turns of the game, is  an $(R,2k)$-expander.
	Here, we think of Maker
	as $d$-Maker in the game $d$-$Box$ with
	appropriate parameters, where the boxes are $\{A_v:v\in V(K_n)\}$. Since $d$-Maker is w.h.p.\ the winner of the $d$-$Box$ game, we will prove that Maker also wins this game.
	
	Consider a game $d$-$Box(n\times s)$ for $d=16k$ and $s=\frac {n-1}2$.
	
	Since $|A_v|= \frac {n-1}2$, by Corollary~\ref{claim:MinDeg} we have that for $m'>20d^2\ln\ln n$ w.h.p.\ $d$-Maker wins the $d$-$Box$ game after at most $\Theta(n\ln n)$ turns of the game. Furthermore, for every box $A_v$, if Maker claimed so far less than $8k$ elements from  $A_v$, then there are at least $\frac n4$ available elements in the box. 
	  For some edge $e$ in Maker's graph, we say that $e=\{u,v\}$
	  was \emph{chosen} by the vertex $v$ if $e\in A_v$.

	According to the setting of the game, in every round Maker claims $C_1 \ln\ln n$ available elements randomly. This can be viewed as follows.
	 In every turn we can think of Maker as choosing first a vertex $v$ with an appropriate probability ( $\frac {f(v,i)}{\sum_{u\in [n]}f(u,i)}$, where  $f(v,i)$ is the number of available elements in $A_v$ just before the $i^{th}$ turn of the game) and then choosing a random element from $A_v$, rather than choosing a random available edge from the graph.

	We now prove that Maker's graph is w.h.p.\ an $(R,2k)$-expander.
	Let $M^*$ be Maker's graph after Maker, as $d$-Maker, wins the $d$-Box game. First, we look in a sub-graph of $M^*$ obtained in the following way. For every set $A_v$ we consider only the first $\frac d2$ elements claimed by $d$-Maker. Then, we look at the graph $L\subseteq M^*$ formed by these edges. It is evident that if $L$ is an $(R,2k)$-expander then $M^*$ is also an $(R,2k)$-expander. Therefore, it is enough to prove now that w.h.p.\ $L$ is an $(R,2k)$-expander. 
	Suppose that the graph $L$ is not an
	$(R,2k)$-expander, then there is a vertex subset $A$, $|A|=a\leq R$, in
	the graph $L$, such that $N_L(A)\subseteq N$, where $|N|=2ka-1$.
	Since Maker claimed from each $A_v$ at least $8k$ edges and $k\geq 1$, we
	can assume that $a\geq 5$ and that there are at least $4ka$ of Maker's
	edges incident to $A$, all of them chosen by vertices from A and all went into $A\cup N$. 
	Assume that at some point
	during the game Maker chose an edge with one vertex $v\in A\cap A_v$ and
	whose second endpoint is in $A\cup N$. This means that there are at least $\frac n4$ available elements in the box
	$A_v$ and therefore at that point of the game, there
	are at least $\frac n4$ unclaimed edges incident to $v$. The
	probability that at that point Maker chose an edge at $v$ whose
	second endpoint belongs to $A\cup N$ is thus at most $\frac{|A\cup
		N|-1}{n/4}$. It follows that the probability that there are at least
	$4ka$ edges chosen by vertices of $A$ that end up in $A\cup N$ is at
	most $\left(\frac {(2k+1)a-2}{n/4}\right)^{4ka}$. 
Thus, the probability that there are at least $4ka$ 
	edges between $A$ and $A\cup N$ is at most
			$\left(\frac {(2k+1)a-2}{n/4}\right)^{4ka} <\left(\frac{12ka}{n}\right)^{4ka}$.
	Therefore the probability that there is such a pair of sets $A,N$ as
	above is at most
	\begin{eqnarray*}
		\sum_{a=5}^{R}\binom{n}{a}\binom{n-a}{2ka-1}\left(\frac
		{12ka}{n}\right)^{4ka}&\leq&
		\sum_{a=5}^{R}\left(\frac{ne}{a}\left(\frac{ne}{2ka}\right)^{2k}\left(\frac
		{12ka}{n}\right)^{4k}\right)^a\\
		&=& \sum_{a=5}^{R}\left(\frac{e^{2k+1}k^{2k}12^{4k}}{2^{2k}}\left(\frac{a}{n}\right)^{2k-1}\right)^a=o(1).\nonumber\\
	\end{eqnarray*}
	The last equality is due to the fact that for $5\leq a\leq \sqrt{n}$
	\[
			\left(\frac{e^{2k+1}k^{2k}12^{4k}}{2^{2k}}\left(\frac{a}{n}\right)^{2k-1}\right)^a
			= \left(\Theta\left(\frac {1}{n^{k-0.5}}\right)\right)^5
			= o\left(\frac1n\right),\]
	
	and for $\sqrt{n}\leq a\leq R$ with $R=\delta n$ where
	$\delta<\left(12ke\right)^{-6}$ is a constant,
	\begin{eqnarray*}
		\left(\frac{e^{2k+1}k^{2k}12^{4k}}{2^{2k}}\left(\frac{a}{n}\right)^{2k-1}\right)^a  &\leq&
		\left(\frac{e^{2k+1}k^{2k}12^{4k}}{2^{2k}}\left(\frac{R}{n}\right)^{2k-1}\right)^{\sqrt{n}}\\
		&=& \left(\frac{e^{2k+1}k^{2k}12^{4k}}{2^{2k}}\cdot\delta^{2k-1}\right)^{\sqrt{n}}\\
		&<& \left({e^{2k+1}k^{2k}12^{4k}}\cdot(12ke)^{-12k+6}\right)^{\sqrt{n}}\\
		&=& o\left(\frac1n\right).
	\end{eqnarray*}
	It follows that Maker is
	able to create an $(R,2k)$-expander w.h.p. in at most $\Theta(\frac {n\ln n}m)$ turns of the game.
\end{proof}

\bigbreak

Using Lemma~\ref{lemma:buildExpander}, we now show that for $m\geq An\ln n$, playing on the edge set of $K_n$, Maker's graph by the end of the random-Maker game contains w.h.p.\ a Hamilton
cycle.

\begin{proof} {\bf of Theorem \ref{thm:RMHam}}
	Let  $d=16$. Let $\delta=\left(13e\right)^{-6}$
	and let $A=21d^2$.
	
 We divide the game into three parts and in each part we show that w.h.p.\ Maker's graph satisfying some ``good" properties:
	
	{\bf Part I -- the graph M is an expander:} Using Lemma \ref{lemma:buildExpander} for  $k=1$ and $R=\delta n$,
	Maker's graph is w.h.p.\ an $R$-expander before both players claimed in total
	$\Theta(n\ln n)$ edges from the graph.

	{\bf Part II -- the graph M is a connected expander:} Denote Maker's graph from Stage I by $M_1$. Then $M_1$ is w.h.p.\ an  $R$-expander, and
	by Lemma \ref{lemma:component} the size of every connected component
	of $M_1$ is w.h.p.\ at least
	$3R$. It follows that there are at most $\frac {n}{3R}=\frac1{3\delta}=:C$ connected components in $M_1$. In the next $2\ln n$ turns the game, Maker's graph will become a connected $R$-expander. Observe that there are at
	least $(3R)^2=9\delta^2n^2$ edges of $K_n$ between any two such
	components.  Also,  in the next $2\ln n$ turns of the game, the number of edges claimed by both  players (together with Stage I) is  $\Theta(n\ln n)$ and therefore the number of available edges between any two connected components during this stage is at least $9\delta^2n^2-\Theta(n\ln n)=\Theta(n^2)$. Denote by $E_1,\dots ,E_{\binom C2}$ the sets of edges between each two connected components.  Maker's goal now is to claim at least one edge from each such set of edges in the next $2\ln n$ turns.  Denote by $\tau$ the number of Maker's turns until he touches every set at least once. The probability of Maker to claim an available edge from some set $E_j$ is at least $\frac {9\delta^2n^2-\Theta(n\ln n)}{n^2}\geq 8\delta^2$ and therefore 
	$$\Pr(\tau>\ln n)\leq \sum_{i=1}^{\binom C2} \left(1-8\delta^2\right)^{\ln n}\leq \binom C2\cdot e^{-8\delta^2\ln n}=o(1).$$
	
	Thus, in the next $2\ln n $ turns of the game, Maker w.h.p.\ is able to claim at least one edge from each set and to make his graph connected. We denote the new graph Maker
	created by $M_2$. It is evident that $M_2$ is still an $R$-expander.
	
	{\bf Part III -- completing a Hamilton cycle:} If $M_2$ contains a
	Hamilton cycle, then we are done. Otherwise, by Lemma
	\ref{lemma:boosters}, $E(K_n)\setminus M_2$ contains at least $\frac {(R+1)^2}2$
	boosters. Observe that after adding a booster, the current graph is
	still an $R$-expander and therefore also contains at least $\frac
	{(R+1)^2}2$ boosters. Clearly, after adding at most $n$ boosters,
	$M_2$ becomes Hamiltonian. We show now that in this stage w.h.p., Maker
	reaches his goal after   $\Theta(n)$ turns of
	the game.
	
	In the next $\frac
	{4n}{\delta^2}$ turns of  Maker, he claims a {\bf random} unclaimed edge  from
	the graph. We are now looking for the probability for such edge to be a booster. There are at most  $\frac{n^2}2-\frac n2$
	unclaimed edges, and according to Lemma \ref{lemma:boosters} there are  at least $\frac{(\delta n+1)^2}{2}$
	boosters in $M_2$. Since by the end of the game Breaker claimed at most $\Theta(n\ln n)$ edges, the number of available boosters after each turn of Maker at any point of this stage is at least $\frac{(\delta n+1)^2}{2}-|E(B)|=\frac{(\delta n+1)^2}{2}-\Theta(n\ln n)>(\delta n/2)^2$. Therefore, in each turn of Maker, until he was able to claim $n$ boosters,
	the  probability for Maker to claim a booster is at least
	$$ \frac {(\delta n/2)^2}{\binom{n}2} \geq \frac {\delta^2}3.$$
	Let $Y$ be the number of boosters Maker claimed in $\frac
	{4n}{\delta^2}$ turns. Then by Lemma
	\ref{Che}, $$\Pr(Y<n)\leq\Pr\left(\Bin(\frac
	{4n}{\delta^2},\frac {\delta^2}3)<n\right) \leq e^{-\left(\frac 14\right)^2\frac43
		n}=o\left(1\right).$$
	Thus in the next $\frac {4n}{\delta^2}$
	turns of Maker he is able w.h.p. to claim at least $n$
	boosters. All in all, in the next $\Theta(n\ln n)$ turns of the game, Maker was  able (w.h.p.) to claim $n$
	boosters and thus Maker's graph in Hamiltonian.
	
\end{proof}

\subsection{Random-Maker $k$-connectivity game}
In this section we prove Theorem~\ref{thm:RMCon}. In order to do so, we first prove that w.h.p.\ also in this game Maker's graph is a good expander (by using Lemma \ref{lemma:buildExpander}) and then show that in $n\ln n$ further turns it becomes $k$-connected.

\begin{proof} Let  $d=16k$, $\delta=(13ke)^{-6}$ and let $A=\max\{21d^2,\frac {9-6\ln \delta}{\delta}\}$. We divide the game into two parts and in each part we show that w.h.p.\ Maker's graphs satisfy some ``good" properties:
	
	{\bf Part I:} Using Lemma \ref{lemma:buildExpander} for $R=\delta n$,
	Maker's graph is w.h.p.\ an $(R,2k)$-expander before both players claimed in total
	$\Theta(n\ln n)$ edges in the graph.

	{\bf Part II:} Maker  makes his graph a $(\frac
	{n+k}{4k},2k)$-expander in $n\ln n$ further turns of
	the game. During the next $n\ln n$ turns of the game, Maker played more than
	$An$ turns for $A\geq \frac {9-6\ln \delta}{\delta}$. It remains to prove that if
	Maker claims $An$ edges randomly, then w.h.p. Maker's graph is a
	$(\frac {n+k}{4k},2k)$-expander. It is enough to prove that
	$E_M(U,W)\neq \emptyset$ for every two subsets $U,W\subseteq V$,
	such that $|U|=|W|=R$. Indeed, if there exists a subset $X\subseteq
	V$ of size $R\leq |X| \leq \frac {n+k}{4k}$ such that $|X\cup
	N_M(X)|<(2k+1)|X|$, then there are two subsets $U\subseteq X$ and
	$W\subseteq V\setminus (X\cup N_M(X))$ such that $|U|=|W|=R$ and
	$E_M(U,W)=\emptyset$. We now prove that w.h.p., $E_M(U,W)\neq
	\emptyset$ for every $|U|=|W|=R$ after $An$ turns of Maker.  Let
	$U,W$ be two subsets such that $|U|=|W|=R$. Recall that in the entire
	game, both players claim at most $\Theta(n\ln n)$ edges.  Thus
	the number of available edges between $U$ and $W$ at any point throughout
	this stage is at least $|U||W|-\Theta(n\ln n)>\frac
	{\delta^2n^2}3$ for a large $n$. Then the probability that Maker
	claims an edge $e\in E(U,W)\setminus E_B(U,W)$ is at least $\frac
	{{\delta^2n^2}/3}{\binom n2}\geq \frac{\delta^2}3$. So at the end of
	Stage II,
	$$\Pr\left(E_M(U,W)=\emptyset\right)=\Pr\left(E(U,W)\setminus E_B(U,W)=\emptyset\right)\leq \left(1-\frac {\delta^2}3\right)^{An}.$$
	Using the
	union bound, we get that the probability that there exist two
	subsets $U,W$, $|U|=|W|=R$ such that $E_M(U,W)=\emptyset$ is at most
	
	\begin{eqnarray*}
		\binom {n}{\delta n} \binom {n}{\delta n}\left(1-\frac{\delta^2}3\right)^{An}&\leq&  \left(\frac{en}{\delta n} \right)^{\delta n}\left(\frac{en}{\delta n} \right)^{\delta n}\left(1-\frac{\delta^2}3\right)^{An}\\
		&\leq& \left(\frac{e}{\delta } \right)^{2\delta n}\left(1-\frac{\delta^2}3\right)^{An}\\
		&\leq& e^{2\delta n-2\delta n\ln {\delta}-\delta^2 An/3}\\
		&=&o(1).
	\end{eqnarray*}
	
	Then w.h.p.\ by Lemma \ref{lemma:connectivity}, since
	$(\frac{n+k}{4k})\cdot 2k\geq \frac 12(|V|+k)$, Maker's graph  is
	$k$-connected and he wins the game.
\end{proof}

\section{Concluding remarks and open problems}

\begin{enumerate}[$(1)$]
	\item 
	In this paper we studied the random-player Maker-Breaker games such as the Hamiltonicity game, the perfect-matching game and the $k$-vertex-connectivity game. In the random-Breaker version, we were able to give asymptotically tight results for the value of the critical bias of the games. We actually proved that the critical bias for these random-Breaker games is asymptotically the maximal value of $b$ that allows Maker's graph to satisfy the desired property. Namely, we proved that for every $\varepsilon>0$ if $b\leq (1-\varepsilon)\frac n2$ then Maker typically wins the random-Breaker Hamiltonicity game, if $b\leq (1-\varepsilon)n$ then Maker typically wins the random-Breaker perfect-matching game and if $b\leq (1-\varepsilon)\frac nk$ then Maker typically wins the random-Breaker $k$-connectivity game.  However, in the random-Maker version, we were only able to determine  the correct order of magnitude  for the critical bias $m^*$. In the random-Maker versions of the above games we proved that $m^*=\Theta(\ln\ln n) $, that is, the maximal value of $m$ that allows Breaker to be the typical winner of the game is of order $\ln\ln n$. One can try to find the exact multiplicative constant in this term.
	\item Note that we only considered random-player  Maker-Breaker games played on $E(K_n)$, and 
	there is still nothing known about random-player games played on different boards. Thus, it would be interesting to study the critical bias for random-player Maker-Breaker games played on different boards, for example, games played on the edge set of some general graph $G$, games played on the edge set of a random graph, games played on the edge set of a hypergraph, etc.
	
	\item With a bit more careful implementation of the same arguments one can also take $\varepsilon=\varepsilon(n)$ to be a concrete vanishing function of $n$, but we decided not to pursue this goal here.
 
\end{enumerate}

 \noindent {\bf Acknowledgment.} The authors would like to thank the referees of the paper for their careful reading and many helpful remarks.

\end{document}